\documentclass[a4paper,12pt]{article}

\usepackage[utf8]{inputenc}
\usepackage[T1]{fontenc}

\usepackage{amsthm,amsmath,amssymb,mathrsfs,bbm,geometry,epsfig,hyperref,comment,nicefrac,listings,graphicx,color,relsize}
\usepackage{makecell}
\usepackage[inline,shortlabels]{enumitem}
\usepackage[capitalise,noabbrev,sort]{cleveref}

\crefname{subsection}{Subsection}{Subsections}
\crefname{equation}{}{}
\crefname{enumi}{item}{items}

\hypersetup{colorlinks=true}

\usepackage{tikz}
\usetikzlibrary{matrix,chains,positioning,decorations.pathreplacing,arrows}
\usetikzlibrary{shapes,arrows}
\tikzset{
  font={\fontsize{9pt}{12}\selectfont}}
\usepackage{adjustbox}

\usepackage{etoolbox} %
\usepackage{listofitems} %
\tikzset{>=latex} %
\colorlet{myred}{red!80!black}
\colorlet{myblue}{blue!80!black}
\colorlet{mygreen}{green!60!black}
\colorlet{mydarkred}{myred!40!black}
\colorlet{mydarkblue}{myblue!40!black}
\colorlet{mydarkgreen}{mygreen!40!black}
\tikzstyle{node}=[very thick,circle,draw=myblue,minimum size=22,inner sep=0.5,outer sep=0.6]
\tikzstyle{connect}=[->,thick,mydarkblue,shorten >=1]
\tikzset{ %
  node 1/.style={node,mydarkgreen,draw=mygreen,fill=mygreen!25},
  node 2/.style={node,mydarkblue,draw=myblue,fill=myblue!20},
  node 3/.style={node,mydarkred,draw=myred,fill=myred!20},
}
\def\nstyle{int(\lay<\Nnodlen?min(2,\lay):3)} %

\usepackage{jlcode}

\def\jlbasicfont{\ttfamily\tiny\selectfont}

\lstdefinestyle{code}{
  language=Julia, 
  showstringspaces=false,
  keywordstyle=\color{blue},
  commentstyle=\color{gray},
  columns=fullflexible,
  keepspaces=true
}
\lstnewenvironment{code}{\lstset{style=code}}{}
\newcommand{\codefile}[1]{\lstinputlisting[style=code]{#1}}
\newcommand{\codelink}[1]{\lstinputlisting[style=code]{#1}
  \noindent\begin{center}
  \filename@parse{#1}
  \href{\codeurl/\filename@base.\filename@ext}
  {\textcolor{blue}{\codelinktext}}
  \end{center}
}
\newcommand{\linkonly}[1]{
  \protect\filename@parse{#1}
  \href{\codeurl/\filename@base.\filename@ext}
  {\textcolor{blue}{\tt[\protect\filename@base.\protect\filename@ext]}}
}

\usepackage{multirow, makecell}

\usepackage{soul,color}

\newtheorem{lemma}{Lemma}[section]

\newtheorem{example}[lemma]{Example}

\newtheorem{algo}[lemma]{Framework}

\DeclareMathOperator{\Hess}{Hess}
\DeclareMathOperator{\Trace}{Trace}

\DeclareMathOperator*{\smallsuml}{\textstyle\sum}
\DeclareMathOperator*{\smallprodl}{\textstyle\prod}
\providecommand{\1}{\mathbbm{1}}
\providecommand{\N}{\mathbb{N}}
\providecommand{\Z}{\mathbb{Z}}
\providecommand{\R}{\mathbb{R}}
\providecommand{\B}{\mathcal{B}}

\renewcommand{\P}{\mathbb{P}}

\providecommand{\bS}{\mathbb{S}}

\renewcommand{\S}{\mathcal{S}}

\providecommand{\D}{D}

\providecommand{\bV}{{\ensuremath{\mathbb{V}}}}

\newcommand{\F}{{\ensuremath{\mathcal{F}}}}

\newcommand{\Y}{\mathcal{Y}}

\newcommand{\Zz}{Z}
\newcommand{\Rr}{{\ensuremath{\mathcal{V}}}}

\newcommand{\diff}{\mathrm{d}}

\newcommand{\bcup}{\bigcup}

\newcommand{\Borel}{\mathcal{B}}

\newenvironment{approxtabular}{
	\tabular{|r|c|c|c|c|c|c|c|c|}
		\hline
		\multirowcell{4}{$d$}
		& \multirowcell{4}{$T$}
		& \multirowcell{4}{$N$}
		& \multirowcell{4}{Mean\\ of the\\ approx.\\ method } 
		& \multirowcell{4}{Standard\\ deviation of \\ the approx.\\ method } 
		& \multirowcell{4}{Reference \\ value } 
		& \multirowcell{4}{Relative\\$L^1$-approx.\\ error } 
		& \multirowcell{4}{Standard\\ deviation \\ of the \\ error } 
		& \multirowcell{4}{Average\\ runtime \\ in \\ seconds } \\
		&&&&&&&&\\
		&&&&&&&&\\
		&&&&&&&&\\
		\hline
}{
	\hline
	\endtabular
}

\usepackage{mathtools}

\DeclarePairedDelimiter{\pr}()
\DeclarePairedDelimiter{\br}[]

\DeclarePairedDelimiter{\abs}\lvert\rvert
\DeclarePairedDelimiter{\norm}\lVert\rVert
\DeclarePairedDelimiter{\ang}\langle\rangle

\newcommand{\bpr}[1]{\pr[\big]{#1}}
\newcommand{\bbpr}[1]{\pr[\Big]{#1}}
\newcommand{\bbbpr}[1]{\pr[\bigg]{#1}}

\newcommand{\bbr}[1]{\br[\big]{#1}}
\newcommand{\bbbr}[1]{\br[\Big]{#1}}
\newcommand{\bbbbr}[1]{\br[\bigg]{#1}}
\newcommand{\bbbbbr}[1]{\br[\Bigg]{#1}}

\newcommand{\babs}[1]{\abs[\big]{#1}}

\newcommand{\bang}[1]{\ang[\big]{#1}}

\begin{document}

\title{Deep learning approximations for non-local nonlinear PDEs with Neumann boundary conditions}

\author{
Victor Boussange$^{1,2}$, 
Sebastian Becker$^{3}$,
Arnulf Jentzen$^{4,5}$,\\
Benno Kuckuck$^{6}$,
and 
Lo{\"i}c Pellissier$^{7,8}$
\bigskip
\\
\small{$^1$ Unit of Land Change Science, Swiss Federal Research Institute}
\vspace{-0.1cm}\\
\small{for Forest, Snow and Landscape (WSL), Switzerland}
\smallskip
\\
\small{$^2$ Landscape Ecology, Institute of Terrestrial Ecosystems,}
\vspace{-0.1cm}\\
\small{Department of Environmental Systems Science, ETH Z\"urich,}
\vspace{-0.1cm}\\
\small{Switzerland, e-mail: \texttt{bvictor}\textcircled{\texttt{a}}\texttt{ethz.ch}}
\smallskip
\\
\small{$^3$ Risklab, Department of Mathematics, ETH Z\"urich,}
\vspace{-0.1cm}\\
\small{Switzerland, e-mail: \texttt{sebastian.becker}\textcircled{\texttt{a}}\texttt{math.ethz.ch}}
\smallskip
\\
\small{$^4$ School of Data Science and Shenzhen Research Institute of Big Data,}
\vspace{-0.1cm}\\
\small{The Chinese University of Hong Kong, Shenzhen, China,}
\vspace{-0.1cm}\\
\small{e-mail: \texttt{ajentzen}\textcircled{\texttt{a}}\texttt{cuhk.edu.cn}}
\smallskip
\\
\small{$^5$ Applied Mathematics: Institute for Analysis and Numerics,}
\vspace{-0.1cm}\\
\small{Faculty of Mathematics and Computer Science, University of M{\"u}nster,}
\vspace{-0.1cm}\\
\small{Germany, e-mail: \texttt{ajentzen}\textcircled{\texttt{a}}\texttt{uni-muenster.de}}
\smallskip
\\
\small{$^6$ Applied Mathematics: Institute for Analysis and Numerics,}
\vspace{-0.1cm}\\
\small{Faculty of Mathematics and Computer Science, University of M{\"u}nster,}
\vspace{-0.1cm}\\
\small{Germany, e-mail: \texttt{bkuckuck}\textcircled{\texttt{a}}\texttt{uni-muenster.de}}
\smallskip
\\
\small{$^7$ Unit of Land Change Science, Swiss Federal Research Institute}
\vspace{-0.1cm}\\
\small{for Forest, Snow and Landscape (WSL), Switzerland}
\smallskip
\\
\small{$^8$ Landscape Ecology, Institute of Terrestrial Ecosystems,}
\vspace{-0.1cm}\\
\small{Department of Environmental System Science, ETH Z\"urich,}
\vspace{-0.1cm}\\
\small{Switzerland, e-mail: \texttt{loic.pellissier}\textcircled{\texttt{a}}\texttt{usys.ethz.ch}}
\smallskip
}

\maketitle
\pagebreak
\begin{abstract}
Nonlinear partial differential equations (PDEs) are used to model dynamical processes in a large number of scientific fields, ranging from finance to biology. In many applications standard local models are not sufficient to accurately account for certain non-local phenomena such as, e.g., interactions at a distance. In order to properly capture these phenomena non-local nonlinear PDE models are frequently employed in the literature. In this article we propose two numerical methods based on machine learning and on Picard iterations, respectively, to approximately solve non-local nonlinear PDEs.
The proposed machine learning-based method is an extended variant of a deep learning-based splitting-up type approximation method previously introduced in the literature and utilizes neural networks to provide approximate solutions on a subset of the spatial domain of the solution. 
The Picard iterations-based method is an extended variant of the so-called \emph{full history recursive multilevel Picard} approximation scheme previously introduced in the literature and provides an approximate solution for a single point of the domain. Both methods are mesh-free and allow non-local nonlinear PDEs with Neumann boundary conditions to be solved in high dimensions.
In the two methods, the numerical difficulties arising due to the dimensionality of the PDEs are avoided by (i) using the correspondence between the expected trajectory of reflected stochastic processes and the solution of PDEs (given by the Feynman--Kac formula) and by (ii) using a plain vanilla Monte Carlo integration to handle the non-local term. 
We evaluate the performance of the two methods on five different PDEs arising in physics and biology.
In all cases, the methods yield good results in up to 10 dimensions with short run times.
Our work extends recently developed methods to overcome the curse of dimensionality in solving PDEs.
\end{abstract}

\pagebreak

\tableofcontents

\newpage
\section{Introduction}

In this article, we derive numerical schemes to approximately solve high-dimensional non-local nonlinear partial differential equations (PDEs) with Neumann boundary conditions. 
Such PDEs have been used to describe  a variety of processes in physics, engineering, finance, and biology, but can generally not be solved analytically, requiring numerical methods to provide approximate solutions. However, traditional numerical methods are for the most part computationally infeasible for high-dimensional problems, calling for the development of novel approximation methods.

The need for solving non-local nonlinear PDEs has been expressed in various fields as they provide a more general description of the dynamical systems than their local counterparts \cite{Kavallaris2018,delia2020numerical,Sunderasan2020}.
In physics and engineering, non-local nonlinear PDEs are found, e.g., in models of Ohmic heating production \cite{Lacey1995}, in the investigation of the fully turbulent behavior of real flows \cite{Caglioti1992}, in phase field models allowing non-local interactions \cite{Barone1971,Gajewski2003,Coleman1994,Hairer2016}, or in phase transition models with conservation of mass \cite{RUBINSTEIN1992,Stoleriu2011}; see \cite{Kavallaris2018} for further references.
In finance, non-local PDEs are used, e.g., in jump-diffusion models for the pricing of derivatives where the dynamics of stock prices are described by stochastic processes experiencing large jumps \cite{Merton1976,Chan1999,Kou2000,Abergel2010,Benth2001,Sunderasan2020,Cruz2020,Tankov2003}. Penalty methods for pricing American put options such as in Kou's jump-diffusion model \cite{Huang2013,Gan2020}, considering large investors where the agent policy affects the assets prices \cite{Amadori2003,Abergel2010}, or considering default risks \cite{Pham2009,Henry-Labordere2012} can further introduce nonlinear terms in non-local PDEs.
In economics, non-local nonlinear PDEs appear, e.g., in evolutionary game theory with the so-called replicator-mutator equation capturing continuous strategy spaces \cite{Oechssler2001,Kavallaris2017,Hamel2020,Alfaro2016,Alfaro2019} or in growth models where consumption is non-local \cite{Banerjee2021}.
In biology, non-local nonlinear PDEs are used, e.g., to model processes determining the interaction and evolution of organisms. Examples include models of morphogenesis and cancer evolution \cite{Lorz2013,Chen2020,Villa2021}, models of gene regulatory networks \cite{Pajaro2017}, population genetics models with the non-local Fisher--Kolmogorov--Petrovsky--Piskunov (Fisher--KPP) equations \cite{FISHER1937,Hamel2001,Bian2017,Perthame2007,Berestycki2009b,Houchmandzadeh2017,Wang2021}, and quantitative genetics models where populations are structured on a phenotypic and/or a geographical space \cite{Burger1994,Genieys2006a,Berestycki2016,Nordbotten2016,Nordbotten2018,Roques2016,Doebeli2010,Nordbotten2020}. In such models, Neumann boundary conditions are used, e.g., to model the effect of the borders of the geographical domain on the movement of the organisms.

Real world systems such as those just mentioned may be of considerable complexity and accurately capturing the dynamics of these systems may require models of high dimensionality \cite{Doebeli2010}, leading to complications in obtaining numerical approximations.
For example, the number of dimensions of the PDEs may correspond in finance to the number of financial assets (such as stocks, commodities, exchange rates, and interest rates) in the involved portfolio; in evolutionary dynamics, to the dimension of the strategy space; and in biology, to the number of genes modelled \cite{Pajaro2017} or to the dimension of the geographical or the phenotypic space over which the organisms are structured.
Standard approximation methods for PDEs such as finite difference approximation methods, finite element methods, spectral Galerkin approximation methods, and sparse grid approximation methods all suffer from the so called \emph{curse of dimensionality} \cite{Bellman1957}, meaning that their computational costs increase exponentially in the number of dimensions of the PDE under consideration. 

Numerical methods exploiting stochastic representations of the solutions of PDEs can in some cases overcome the curse of dimensionality. Specifically, simple Monte Carlo averages of the associated stochastic processes have been proposed a long time ago to solve high-dimensional linear PDEs, such as, e.g., Black--Scholes and Kolmogorov PDEs \cite{metropolis1949monte,bauer1958monte}. 
Recently, two novel classes of methods have proved successful in dealing with high-dimensional nonlinear PDEs, namely deep learning-based and full history recursive multilevel Picard approximation methods (in the following we will abbreviate \emph{full history recursive multilevel Picard} by MLP).
The explosive success of deep learning in recent years across a wide range of applications \cite{LeCun2015} has inspired a variety of neural network-based approximation methods for high-dimensional PDEs; see \cite{Beck2020} for an overview. One class of such methods is based on reformulating the PDE as a stochastic learning problem through the Feynman--Kac formula \cite{EHanJentzen2017,Han2018,beck2018solving}. In particular, the \emph{deep splitting} scheme introduced in \cite{Beck2019} relies on splitting the differential operator into a linear and a nonlinear part and in that sense belongs to the class of splitting-up methods \cite{cox2013pathwise,gyongy2003splitting,hochbruck2005explicit}. The PDE approximation problem is then decomposed along the time axis into a sequence of separate learning problems. The deep splitting approximation scheme has proved capable of computing reasonable approximations to the solutions of nonlinear PDEs in up to 10000 dimensions.
On the other hand, the MLP approximation method, introduced in \cite{Weinan2019,Hutzenthaler2020,Weinan2021}, utilizes the Feynman--Kac formula to reformulate the PDE problem as a fixed point equation. It further reduces the complexity of the numerical approximation of the time integral through a multilevel Monte Carlo approach.
However, neither the deep splitting nor the MLP method can, until now, account for non-localness and Neumann boundary conditions.

The goal of this article is to overcome these limitations and thus we generalize the deep splitting method and the MLP approximation method to approximately solve non-local nonlinear PDEs with Neumann boundary conditions. We handle the non-local term by a plain vanilla Monte Carlo integration and address Neumann boundary conditions by constructing reflected stochastic processes. While the MLP method can, in one run, only provide an approximate solution at a single point $x \in D$ of the spatial domain $D\subseteq\R^d$ where $d\in\N=\{1,2,\dots\}$, the machine learning-based method can in principle provide an approximate solution on a full subset of the spatial domain $D$ (however, cf., e.g., \cite{heinrich1998monte,heinrich1999monte,grohs2021proof} for results on limitations on the performance of such approximation schemes).
We use both methods to solve five non-local nonlinear PDEs arising in models from biology and physics and cross-validate the results of the simulations.
We manage to solve the non-local nonlinear PDEs with reasonable accuracy in up to 10 dimensions.

For an account of classical numerical methods for solving non-local PDEs, such as finite differences, finite elements, and spectral methods, we refer the reader to the recent survey \cite{delia2020numerical}. Several machine-learning based schemes for solving non-local PDEs can also be found in the literature. 
In particular, the \emph{physics-informed neural network} and \emph{deep Galerkin} approaches \cite{raissi2019physics,Sirignano2018dgm}, based on representing an approximation of the whole solution of the PDE as a neural network and using automatic differentiation to do a least-squares minimization of the residual of the PDE, have been extended to fractional PDEs and other non-local PDEs \cite{pang2019fpinns,lu2021deepxde,guo2022monte,alaradi2019extensions,yuan2022apinn}. While some of these approaches use classical methods susceptible to the curse of dimensionality for the non-local part \cite{pang2019fpinns,lu2021deepxde}, mesh-free methods suitable for high-dimensional problems have also been investigated \cite{guo2022monte,alaradi2019extensions,yuan2022apinn}. 

The literature also contains approaches that are more closely related to the machine learning-based algorithm presented here. Frey \& K\"ock \cite{frey2021deep,frey2022deep} propose an approximation method for non-local semilinear parabolic PDEs with Dirichlet boundary conditions based on and extending the deep splitting method in \cite{Beck2019} and carry out numerical simulations for example PDEs in up to 4 dimensions. Castro \cite{castro2021deep} proposes a numerical scheme for approximately solving non-local nonlinear PDEs based on \cite{hure2020deep} and proves convergence results for this scheme.
Finally, Gonon \& Schwab \cite{gonon2021deep} provide theoretical results showing that neural networks with ReLU activation functions have sufficient expressive power to approximate solutions of certain high-dimensional non-local linear PDEs without the curse of dimensionality.

There is a more extensive literature on machine learning-based methods for approximately solving standard PDEs without non-local terms but with various boundary conditions, going back to early works by Lagaris et al.~\cite{lagaris1998artificial,lagaris2000neural} (see also \cite{mcfall2009artificial}), which employed a grid-based method based on least-squares minimization of the residual and shallow neural networks to solve low-dimensional ODEs and PDEs with Dirichlet, Neumann, and mixed boundary conditions. More recently, approximation methods for PDEs with Neumann (and other) boundary conditions have been proposed using, e.g., physics-informed neural networks \cite{lu2021deepxde,sukumar2022exact,WangPerdikaris2020}, the \emph{deep Ritz} method (based on a variational formulation of certain elliptic PDEs) \cite{e2018deep,liao2021deep,chen2020comparison}, or adversarial networks \cite{zang2020weak}.

The remainder of this article is organized as follows. \Cref{sec:derivation_spec}  discusses a special case of the proposed machine learning-based method, in order to provide a readily comprehensible exposition of the key ideas of the method. \Cref{sec:derivation_gen} discusses the general case, which is flexible enough to cover a larger class of PDEs and to allow more sophisticated optimization methods. \Cref{sec:MLP} presents our extension of the MLP approximation method to non-local nonlinear PDEs, which we use to obtain reference solutions in \cref{sec:examples}. \Cref{sec:examples} provides numerical simulations for five concrete examples of (non-local) nonlinear PDEs. \Cref{sec:sourcecodes} provides the source codes used for the computations in \cref{sec:examples}.

\section{Machine learning-based approximation method in a special case}
\label{sec:derivation_spec}

In this section, we present in \cref{algo:special_case} in \cref{subsec:algo_spec} below a simplified version of our general machine learning-based algorithm for approximating solutions of non-local nonlinear PDEs with Neumann boundary conditions proposed in \cref{sec:derivation_gen} below. This simplified version applies to a smaller class of non-local heat PDEs, specified in \cref{subsec:pde_spec} below. In \cref{subsec:temp-discret} we introduce some notation related to the reflection of straight lines on the boundaries of a suitable subset $D\subseteq\R^d$ where $d\in\N$, which will be used to describe time-discrete reflected stochastic processes that are employed in our approximations throughout the rest of the article. The simplified algorithm described in \cref{subsec:algo_spec} below is limited to using neural networks of a particular architecture that are trained using plain vanilla stochastic gradient descent, whereas the full version proposed in \cref{def:general_algorithm} in \cref{subsec:algo-Full-gen} below is formulated in such a way that it encompasses a wide array of neural network architectures and more sophisticated training methods, in particular Adam optimization, minibatches, and batch normalization. Stripping away some of these more intricate aspects of the full algorithm is intended to exhibit more acutely the central ideas in the proposed approximation method.

The simplified algorithm described in this section as well as the more general version proposed in \cref{def:general_algorithm} in \cref{subsec:algo-Full-gen} below are based on the deep splitting method introduced in Beck et al.~\cite{Beck2019}, which combines operator splitting with a previous deep learning-based approximation method for Kolmogorov PDEs \cite{beck2018solving}; see also Beck et al.~\cite[Sections~2 and~3]{Beck2020} for an exposition of these methods.

\subsection{Partial differential equations (PDEs) under consideration}
\label{subsec:pde_spec}

Let
	$ T \in (0,\infty) $,
	$ d \in \N $,
let
	$\D \subseteq \R^d$
	be a closed set with sufficiently smooth boundary	$\partial_\D$,
let 
	$ \mathbf{n} \colon \partial_\D \to \R^d $ 
	be an outer unit normal vector field associated to $\D$,
let
	$g\in C(D, \R)$,
let
	$\nu_x \colon \Borel(\D) \to [0,1]$,
	$x \in \D$,
	be probability measures,
let
	$f \colon \R \times \R \to \R$
	be measurable,
let
	$u=(u(t,x))_{(t,x)\in [0,T]\times\D}\in C^{1,2}([0,T]\times\D,\R)$
	have at most polynomially growing partial derivatives, 
assume\footnote{Throughout this article we denote
by 
	$\ang{\cdot,\cdot}\colon\pr*{\bigcup_{n\in\N} (\R^n\times\R^n)}\to\R$
	and $\norm{\cdot}\colon\pr*{\bigcup_{n\in\N}\R^n}\to\R$ the 
functions which satisfy for all 
	$n\in\N$, 
	$v=(v_1,\dots,v_n),\,w=(w_1,\dots,w_n)\in\R^n$
that
	$\ang{v,w}=\sum_{i=1}^n v_iw_i$
and $\norm{v}=\sqrt{\ang{v,v}}=\bbr{\sum_{i=1}^n \abs{v_i}^2}^{\nicefrac{1}{2}}$.
}
	for every
		$t\in (0,T]$,
		$x\in \partial_\D$
	that
		$ \ang*{\mathbf{n}(x) ,(\nabla_x u)(t,x)} = 0$,
and assume 
	for every
		$t\in [0,T]$,
		$x\in\D$
	that
		$u(0,x)=g(x)$,
		$\int_\D \abs{f(u(t,x),u(t,\mathbf{x})) } \, \nu_x(\diff\mathbf{x}) < \infty$,
		and
	\begin{equation}
	\begin{split}
	\label{eq:defPDEspecial}
		\bpr{\tfrac{\partial}{\partial t}u}(t,x)
		&=
		(\Delta_x u)(t,x)+\int_{\D} f(u(t,x),u(t,\mathbf{x})) \, \nu_x(\diff\mathbf{x})
		.
	\end{split}
	\end{equation}

Our goal in this section is to approximately calculate under suitable hypotheses the solution $u\colon [0,T]\times \D \to \R$ of the PDE in \eqref{eq:defPDEspecial}.

\subsection{Reflection principle for the simulation of time discrete reflected processes}
\label{subsec:temp-discret}
\begin{algo}[Reflection principle for the simulation of time discrete reflected processes]
	\label{algo:time_discrete_reflected_processes}
	Let
		$d \in \N$, %
	let
		$\D \subseteq \R^{ d }$
		be a closed set with sufficiently smooth boundary	$\partial_\D$,
	let
		$ \mathbf{n} \colon \partial_\D \to \R^d $ 
		be a suitable outer unit normal vector field associated to $\D$,
	let 
		$\mathfrak{c} \colon (\R^d)^2  \to \R^d $ 
	satisfy 
		for every
			$a,b \in \R^d$
		that
		\begin{equation}
			\mathfrak{c}(a,b) 
			=
			a + \bbr{\inf( \{r \in [0,1] \colon a + r(b - a) \notin \D\}\cup \{1\})} (b-a)
			,
		\end{equation}
	let
		$\mathscr{R} \colon (\R^{ d })^2 \to ( \R^{ d })^2 $
	satisfy 
		for every
			$a,b \in \R^d$
		that
		\begin{equation}
			\mathscr{R}(a,b) 
			= 
			\begin{cases}
				(a,b)
				& \colon \mathfrak{c}(a,b) =a \\
				\bpr{ \mathfrak{c}(a,b), b - 2\mathbf{n}(\mathfrak{c}(a,b)) \bang{ b-\mathfrak{c}(a,b),\mathbf{n}(\mathfrak{c}(a,b))} }
				& \colon \mathfrak{c}(a,b) \notin\{a, b\} \\
				(b,b) 
				& \colon \mathfrak{c}(a,b) = b,
			\end{cases}
		\end{equation}
	let
		$P \colon (\R^d)^2 \to \R^d$
	satisfy 
		for every
			$a,b \in \R^d$
		that
			$P(a,b) = b$,
	let
		$\mathcal{R}_{n} \colon (\R^d)^2 \to ( \R^{ d } )^2$, $n \in \N_0 = \{0\} \cup \N$,
	satisfy 
		for every
			$n \in \N_0 $,
			$x,y \in \R^d$
		that
			$\mathcal{R}_0(x,y) = (x,y)$ and
			$\mathcal{R}_{ n + 1 }(x,y) = \mathscr{R}( \mathcal{R}_{n} (x,y ) ) $,
	and let
		$R \colon (\R^d)^2 \to \R^d $
	satisfy 
		for every
			$x,y\in \R^d$
		that
		\begin{equation}
			R(x,y)
			=
			{\textstyle \lim_{n\to \infty}} P(\mathcal{R}_{n}(x,y))
			.
		\end{equation}
\end{algo}

\subsection{Description of the proposed approximation method in a special case}
\label{subsec:algo_spec}

\begin{algo}[Special case of the machine learning-based approximation method]
	\label{algo:special_case}
	Assume 
		\cref{algo:time_discrete_reflected_processes},
	let
		$T,\gamma\in (0,\infty)$,
		$N,M,K \in\N$, 
		$g \in C^2(\R^d,\R)$,
		$\mathfrak d,\mathfrak{h} \in \N \backslash \{1\}$,
		$t_0,t_1,\ldots,t_N\in [0,T]$ 
	satisfy
		$\mathfrak{d} = \mathfrak{h}(N+1)d(d+1)$
	and
	\begin{equation}
		\label{eq:algo_spec_times}
		0 = t_0 < t_1 < \ldots < t_N = T
		,
	\end{equation}
	let 
		$\tau_0, \tau_1, \dots,\tau_N \in [0,T]$ 
	satisfy 
		for every 
			$n \in \{0,1,\dots,N\}$ 
		that 
			$\tau_n= T-t_{N-n}$,
	let
		$f \colon \R \times \R \to \R$
		be measurable,
	let
		$(\Omega,\F,\P,(\mathcal{F}_t)_{t\in [0,T]})$
		be a filtered probability space,
	let
		$\xi^{m}\colon\Omega\to\R^d$, $m\in\N$,
		be i.i.d.\ $\mathcal{F}_0$/$\B(\R^d)$-measurable random variables,
	let
		$W^{m}\colon [0,T]\times \Omega \to\R^d$, $m\in\N$, 
		be i.i.d.\ standard	$(\mathcal{F}_t)_{t\in [0,T]}$-Brownian motions,
	for every
		$m\in\N$
	let
		$\Y^{m}\colon \{0,1,\ldots,N\}\times\Omega\to\R^d$
		be the stochastic process 
		which satisfies 
			for every 
				$n\in\{0,1,\ldots,N-1\}$   
			that
				$\Y^{m}_0 = \xi^{m}$ and
				\begin{equation}\label{Y-algo-spez}
					\Y^{m}_{n+1}
					=
					R\bpr{\Y^{m}_{n}, \Y^{m}_{n} + \sqrt{2} (W^{m}_{\tau_{n+1}}-W^{m}_{\tau_{n}})}
					,
				\end{equation}
	let 
		$ \mathcal{L} \colon \R^d \to \R^d $ 
	satisfy 
		for every 
			$ x = ( x_1, \dots, x_d ) \in \R^d $ 
		that
		\begin{equation}
			\label{eq:activation}
			\mathcal{L}( x )
			=
			\pr*{\frac{\exp(x_1)}{\exp(x_1)+1},\dots,\frac{\exp(x_d)}{\exp(x_d)+1}}
			,
		\end{equation}
	for every 
		$ \theta = ( \theta_1, \dots, \theta_{ \mathfrak{d} } ) \in \R^{ \mathfrak{d} }$,
		$k, l, v \in \N $
		with
			$v + l (k + 1 ) \leq \mathfrak{d}$
	let
		$ A^{ \theta, v }_{ k, l } \colon \R^k \to \R^l $
	satisfy 
		for every
			$ x = ( x_1, \dots, x_k )\in\R^k $ 
		that
		\begin{equation}\label{eq:layerA}
			A^{ \theta, v }_{ k, l }( x ) 
			= 
			\bbbpr{
				\theta_{v+kl+1}+ \br*{\textstyle\sum\limits_{i=1}^ k x_i\, \theta_{v+i}}, 
				\dots, 
				\theta_{v+kl+l} + \br*{\textstyle\sum\limits_{i=1}^ k x_i\, \theta_{v+(l-1)k+i}} 
			}
			,
		\end{equation}	%
	let
		$\bV_n\colon\R^{\mathfrak{d}}\times\R^d\to\R$, $n\in\{0,1,\ldots,N\}$,
	satisfy 
		for every
			$n\in\{1,2,\ldots,N\}$, 
			$\theta \in \R^{\mathfrak{d}}$, 
			$x \in \R^d$
		that
			$\bV_0(\theta,x) = g(x)$ and
			\begin{align}
				\label{eq:neural_network_for_a_generalCase}
				& \bV_{ n }(\theta,x)
				= \\
				& \bpr{A^{ \theta, (\mathfrak{h}n+\mathfrak{h}-1)d(d+1) }_{ d, 1 }
				\circ
				\mathcal{L}
				\circ
				A^{ \theta, (\mathfrak{h}n+\mathfrak{h}-2)d(d+1) }_{ d, d }
				\circ
				\ldots
				\circ
				\mathcal{L}
				\circ
				A^{ \theta, (\mathfrak{h}n+1)d(d+1) }_{ d, d }
				\circ
				\mathcal{L}
				\circ
				A^{ \theta, \mathfrak{h}nd(d+1) }_{ d, d }}(x)
				, \nonumber
			\end{align}
	let
		$\nu_x \colon \mathcal{B}( \D )\to [0,1]$,
		$x \in \D$,
		be probability measures, 
	for every
		$x \in \D$
	let
		$Z^{ n, m }_{ x, k } \colon \Omega \to \D$,
		$k, n, m \in \N$,
		be i.i.d.\ random variables 
		which satisfy 
			for every
				$ A \in \mathcal{B}( \D ) $
			that
				$\P( Z_{ x, 1 }^{ 1, 1 } \in A ) = \nu_x( A )$,
	let
		$\Theta^{n}\colon \N_0\times\Omega\to\R^{\mathfrak{d}}$, $n \in \{0,1,\ldots,N\}$,
		be stochastic processes,
	for every 
		$n \in \{1,2,\ldots,N\}$, 
		$m\in \N$
	let
		$\phi^{n,m}\colon\R^{\mathfrak{d}}\times\Omega\to\R$
	satisfy 
		for every
			$\theta\in\R^{\mathfrak{d}}$, $\omega\in\Omega$
		that
		\begin{equation}
		\label{eq:loss_special_case}
		\begin{split}
			&\phi^{n,m}(\theta,\omega)
			= 
			\bbbbr{ \bV_n\bpr{\theta,\Y^{m}_{N-n}(\omega)} - \bV_{n-1}\bpr{\Theta^{n-1}_M(\omega),\Y^{m}_{N-n+1}(\omega)}\\
			& - \tfrac{(t_{n}-t_{n-1})}{K} \bbbbr{ \smallsuml_{k=1}^{K}  f\bpr{\bV_{n-1}(\Theta^{n-1}_M(\omega),\Y^{m}_{N-n+1}(\omega)),\bV_{n-1}(\Theta^{n-1}_M(\omega),Z_{ \mathcal{Y}^m_{ N - n + 1 }(\omega), k }^{ n, m }(\omega)) } } }^2, %
		\end{split}
		\end{equation}
	for every
		$n\in\{1,2,\ldots,N\}$,
		$m\in\N$
	let
		$\Phi^{n,m}\colon\R^{\mathfrak{d}}\times\Omega\to\R^{\mathfrak{d}}$
	satisfy 
		for every
			$\theta\in\R^{\mathfrak{d}}$,
			$\omega\in\Omega$
		that
			$\Phi^{n,m}(\theta,\omega) = (\nabla_{\theta}\phi^{n,m})(\theta,\omega)$,
	and assume for every
		$n\in\{1,2,\ldots,N\}$,
		$m\in\N$
	that
	\begin{equation}
	\label{eq:plain-vanilla-SGD}
		\Theta^{n}_{m} 
		=
		\Theta^{n}_{m-1} - \gamma\,\Phi^{n,m}(\Theta^{n}_{m-1})
		.
	\end{equation}
\end{algo}
As indicated in \cref{subsec:pde_spec} above, the algorithm described in
\cref{algo:special_case} computes 
an approximation for a solution of the PDE in
\cref{eq:defPDEspecial}, i.e., a
function
$u\in C^{1,2}([0,T]\times D,\R)$
which has at most polynomially growing derivatives,
which satisfies for every
$t\in (0,T]$,
$x\in \partial_\D$
that
$ \ang*{\mathbf{n}(x) ,(\nabla_x u)(t,x)} = 0$
and which satisfies for every
$t\in [0,T]$,
$x\in\D$
that
$u(0,x)=g(x)$,
$\int_\D \abs{f(u(t,x), \allowbreak u(t,\mathbf{x})) } \, \nu_x(\diff\mathbf{x}) < \infty$,
and
\begin{equation}
\begin{split}
\label{eq:defPDEspecial2}
\bpr{\tfrac{\partial}{\partial t}u}(t,x)
&=  (\Delta_x u)(t,x)
+
\int_{\D} f(u(t,x),u(t,\mathbf{x})) \, \nu_x(\diff\mathbf{x}).
\end{split}
\end{equation}
Let us now add some explanatory comments on the objects and
notations employed in \cref{algo:special_case} above.
The algorithm in \cref{algo:special_case} decomposes the time interval $[0,T]$ into $N$ subintervals at the times $t_0,t_1,t_2,\dots,t_{N}\in[0,T]$ (cf.~\cref{eq:algo_spec_times}). For every $n\in\{1,2,\dots,N\}$ we aim to approximate the function $\R^d\ni x\mapsto u(t_n,x)\in \R$ by a suitable (realization function of a) fully-connected feedforward neural network. Each of these neural networks is an alternating composition of $\mathfrak h-1$ affine linear functions from $\R^d$ to $\R^d$ (where we think of $\mathfrak h\in\N\backslash\{1\}$ as the \emph{length} or \emph{depth} of the neural network), $\mathfrak h-1$ instances of a $d$-dimensional version of the standard logistic function and finally an affine linear function from $\R^d$ to $\R$. Every such neural network can be specified by means of $(\mathfrak h-1)(d^2+d)+d+1\leq \mathfrak hd(d+1)$ real parameters and so $N+1$ of these neural networks can be specified by a parameter vector of length $\mathfrak d=\mathfrak h(N+1)d(d+1)\in\N$.
Note that $\mathcal L\colon \R^d\to\R^d$ in \cref{algo:special_case} above denotes the $d$-dimensional version of the standard logistic function (cf.~\cref{eq:activation}) and for every $k,l,v\in\N$, $\theta\in\R^{\mathfrak d}$ with $v+kl+l\leq\mathfrak d$ the function $A^{\theta,v}_{k,l}\colon \R^k\to\R^l$ in \cref{algo:special_case} denotes an affine linear function specified by means of the parameters $v+1,v+2,\dots,v+kl+l$ (cf.~\cref{eq:layerA}). Furthermore, observe that for every $n\in\{1,2,\dots,N\}$, $\theta\in\R^{\mathfrak d}$ the function 
\begin{equation}
\R^d\ni x\mapsto \mathbb V_n(\theta,x)\in\R
\end{equation}
denotes a neural network specified by means of the parameters $\mathfrak hnd(d+1)+1, \mathfrak hnd(d+1)+2,\dots,(\mathfrak hn+\mathfrak h-1)d(d+1)+d+1$. 

The goal of the optimization algorithm in \cref{algo:special_case} above is to find a suitable parameter vector $\theta\in\R^{\mathfrak d}$ such that for every $n\in\{1,2,\dots,N\}$ the neural network $\R^d\ni x\mapsto \mathbb V_n(\theta,x)\in\R$ is a good approximation for the solution $\R^d\ni x\mapsto u(t_n,x)\in \R$ to the PDE in \cref{eq:defPDEspecial2} at time $t_n$.
This is done by performing successively for each $n\in\{1,2,\dots,N\}$ a plain vanilla stochastic gradient descent (SGD) optimization on a suitable loss function (cf.~\cref{eq:plain-vanilla-SGD}). 

Observe that for every $n\in\{1,2,\dots,N\}$ the stochastic process $\Theta^n\colon \N_0\times \Omega\to\R^{\mathfrak d}$ describes the successive estimates computed by the SGD algorithm for the parameter vector that represents (via $\mathbb V_n\colon \R^{\mathfrak d}\times\R^d\to\R$) a suitable approximation to the solution $\R^d\ni x\mapsto u(t_n,x)\in \R$ of the PDE in \cref{eq:defPDEspecial2} at time $t_n$.
Next note that $M\in\N$ in \cref{algo:special_case} above denotes the number of gradient descent steps taken for each $n\in\{1,2,\dots,N\}$ and that $\gamma\in(0,\infty)$ denotes the learning rate employed in the SGD algorithm. Moreover, observe that for every $n\in\{1,2,\dots,N\}$, $m\in\{1,2,\dots,M\}$ the function $\phi^{n,m}\colon \R^{\mathfrak d}\times\Omega\to\R$ denotes the loss function employed in the $m$th gradient descent step during the approximation of the solution of the PDE in \cref{eq:defPDEspecial2} at time $t_n$ (cf.~\cref{eq:loss_special_case}). The loss functions employ a family of i.i.d.\ time-discrete stochastic processes $\mathcal Y^m\colon \{0,1,\dots N\}\times\Omega\to\R^d$, $m\in\N$, which we think of as discretizations of suitable reflected Brownian motions (cf.~\cref{Y-algo-spez}). In addition, for every $n\in\{1,2,\dots,N\}$, $m\in\{1,2,\dots,M\}$, $x\in D$ the loss function $\phi^{n,m}\colon \R^{\mathfrak d}\times\Omega\to\R$ employs a family of i.i.d.~random variables $Z^{n,m}_{x,k}\colon \Omega\to D$, $k\in\N$, which are used for the Monte Carlo approximation of the non-local term in the PDE in \cref{eq:defPDEspecial2} whose solution we are trying to approximate. The number of samples used in these Monte Carlo approximations is denoted by $K\in\N$ in \cref{algo:special_case} above.

Finally, for sufficiently large $N,M,K\in\N$ and sufficiently small $\gamma\in(0,\infty)$ the algorithm in \cref{algo:special_case} above yields for every $n\in\{1,2,\dots,N\}$ a (random) parameter vector $\Theta^n_M\colon \Omega\to\R^{\mathfrak d}$
which represents a function $\R^d\times\Omega\ni (x,\omega)\mapsto \mathbb V_n(\Theta^n_M(\omega),x)\in \R$ that we think of as providing for every $x\in\D$ a suitable approximation
\begin{equation}
	\mathbb V_n(\Theta^n_M,x) \approx u(t_n,x)
	.
\end{equation}

\section{Machine learning-based approximation method in the general case}
\label{sec:derivation_gen}

In this section we describe in \cref{def:general_algorithm} in \cref{subsec:algo-Full-gen} below the full version of our deep learning-based method for approximating solutions of non-local nonlinear PDEs with Neumann boundary conditions (see \cref{subsec:gen_pdes} for a description of the class of PDEs our approximation method applies to), which generalizes the algorithm introduced in \cref{algo:special_case} in \cref{subsec:algo_spec} above and which we apply in \cref{sec:examples} below to several examples of non-local nonlinear PDEs.

\subsection{PDEs under consideration}
\label{subsec:gen_pdes}
Let
$ T \in (0,\infty) $,
$ d \in \N $,
let
$\D \subseteq \R^d$
be a closed set with sufficiently smooth boundary
$\partial_\D$,
let $ \mathbf{n} \colon \partial_\D \to \R^d $ be an outer unit normal vector field associated to $\D$,
let
$
g\colon \D \to \R
$,
$
\mu \colon \D \to \R^d
$,
and
$
\sigma \colon \D \to \R^{ d \times d }
$
be
continuous, let
$\nu_x \colon \Borel(\D) \to [0,1]$,
$x \in \D$,
be probability measures,
let
$
f \colon [0,T] \times \D \times \D \times \R \times \R
\to \R
$
be measurable,
let
$
u=
(u(t,x))_{(t,x)\in [0,T]\times\D}\in C^{1,2}([0,T]\times\D,\R)
$
have at most polynomially growing partial derivatives, 
assume for every
$t\in [0,T]$,
$x\in \partial_\D$
that
$ \ang*{\mathbf{n}(x) ,(\nabla_x u)(t,x)} = 0$,
and assume for every
$t\in [0,T]$,
$x\in\D$
that
$u(0,x)=g(x)$,
$\int_\D \babs{f\bpr{t,x,\mathbf{x}, u(t,x),u(t,\mathbf{x})
} } \allowbreak 
\, \nu_x(\diff\mathbf{x}) < \infty$,
and
\begin{equation}
\begin{split}
\bpr{\tfrac{\partial}{\partial t}u}(t,x)
&=
\int_{\D} f\bpr{t,x,\mathbf{x}, u(t,x),u(t,\mathbf{x})
} \, \nu_x(\diff\mathbf{x}) \\
& \quad + \bang{\mu(x), ( \nabla_x u )( t,x ) }
+ \tfrac{ 1 }{ 2 }
\Trace\bpr{
\sigma(x) [ \sigma(x) ]^*
( \Hess_x u)( t,x )
}.
\label{eq:defPDE}
\end{split}
\end{equation}

Our goal is to approximately calculate under suitable hypotheses the solution $u\colon [0,T]\times \D \to \R$ of the PDE in \eqref{eq:defPDE}.

\subsection{Description of the proposed approximation method in the general case}
\label{subsec:algo-Full-gen}

\begin{algo}[General case of the machine learning-based approximation method]
	\label{def:general_algorithm}
	Assume 
		\cref{algo:time_discrete_reflected_processes},
	let
		$T \in (0,\infty)$,
		$N, \varrho, \mathfrak{d}, \varsigma \in \N$,
		$(M_n)_{n\in\N_0}\subseteq\N$,
		$(K_n)_{n\in \N}\subseteq\N$,
		$(J_m)_{m \in \N} \subseteq \N$,
		$t_0,t_1,\ldots,t_N\in [0,T]$ 
	satisfy
	\begin{equation}
		0 = t_0 < t_1 < \ldots < t_N = T
		,
	\end{equation}
	let 
		$\tau_0, \tau_1, \dots,\tau_N \in [0,T]$ 
	satisfy 
		for every 
			$n \in \{0, 1, \dots, N\}$ 
		that 
			$\tau_n= T-t_{N-n}$,
	let
		$\nu_x \colon \mathcal{B}( \D )\to [0,1]$,
		$x \in \D$,
		be probability measures, 
	for every
		$x \in \D$
	let
		$Z^{ n, m,j }_{ x, k } \colon \Omega \to \D$,
		$k, n, m, j \in \N$,
		be i.i.d.\ random variables 
		which satisfy 
			for every
				$ A \in \mathcal{B}( \D ) $
			that
				$\P( Z_{ x, 1 }^{ 1, 1,1 } \in A ) = \nu_x( A )$,
	let
		$f\colon [0,T] \times \D \times \D \times \R \times \R 
		\to \R$
	be measurable,
	let
		$( \Omega, \F, \P, ( \mathcal{F}_t )_{ t \in [0,T] } )$
		be a filtered probability space,
	for every
		$n \in \{1,2,\ldots,N\}$
	let
		$W^{n,m,j} \colon [0,T] \times \Omega \to \R^d$,
		$m,j \in \N$,
		be i.i.d.\ standard $( \mathcal{F}_t )_{ t \in [0,T] }$-Brownian motions,
	for every
	    $n \in \{1, 2, \ldots,N\}$
	let
		$\xi^{n,m,j}\colon\Omega\to\R^d$,
		$m, j \in \N $,
		be i.i.d.\ $ \mathcal{F}_0$/$\B(\R^d)$-measurable random variables,
	let
		$H\colon [0,T]^2\times\R^d\times\R^d\to\R^d$
		be a function,
	for every
		$j\in\N$,
		$\mathbf{s}\in \R^\varsigma$,
		$n\in\{0, 1, \dots,N\}$
	let
		$ \bV^{j,\mathbf{s}}_n\colon \R^{\mathfrak{d}}\times\R^d\to \R$ be a function,
	for every
		$ n \in \{1,2, \ldots,N\}$,
		$ m,j \in \N $
	let
		$\Y^{n,m,j}\colon \{0,1,\ldots,N\}\times\Omega\to\R^d$
		be a stochastic process which satisfies 
			for every 
				$k\in\{0,1,\ldots,N-1\}$ 
			that 
				$\Y^{n,m,j}_0 = \xi^{ n, m, j }$ and
			\begin{equation}\label{eq:FormalXapprox}
				\Y^{n,m,j}_{k+1}
				=
				H(\tau_{k+1},\tau_{k},\Y^{n,m,j}_k,W^{n,m,j}_{\tau_{k + 1}} - W^{n,m,j}_{\tau_{k}})
				,
			\end{equation}
	let 
		$\Theta^{n}\colon\N_0\times\Omega\to\R^{\mathfrak{d}}$, $n \in \{0, 1, \dots, N\}$, 
		be stochastic processes,
	for every
		$n\in\{1,2,\ldots,N\}$,
		$m\in\N$,
		$\mathbf{s}\in\R^{\varsigma}$
	let
		$\phi^{n,m,\mathbf{s}}\colon\R^{\mathfrak{d}}\times\Omega\to\R$
	satisfy 
		for every
			$\theta\in\R^{\mathfrak{d}}$,
			$\omega\in\Omega$ 
		that
		\begin{equation}
			\label{eq:loss_general_case}
		\begin{split}
			& \phi^{n,m,\mathbf{s}}(\theta,\omega)
			=
			\frac{1}{J_m}\sum_{j=1}^{J_m}
			\bbbbbr{
			\bV^{j,\mathbf{s}}_n\bpr{\theta,\Y^{n,m,j}_{N-n}(\omega)}
			-
			\bV^{j,\mathbf{s}}_{n-1}\bpr{\Theta^{n-1}_{M_{n-1}}(\omega),\Y^{n,m,j}_{N-n+1}(\omega)}\\
			& - \tfrac{(t_n-t_{n-1})}{K_n} \bbbbr{ \textstyle \sum \limits_{k=1}^{K_n}  f\bbpr{t_{n-1},
			\Y^{n,m,j}_{N-n+1}(\omega),
			Z_{ \mathcal{Y}^{n,m,j}_{ N - n + 1 }(\omega), k }^{ n, m,j }(\omega),\\
			& \bV^{j,\mathbf{s}}_{n-1}\bpr{\Theta^{n-1}_{M_{n-1}}(\omega),\Y^{n,m,j}_{N-n+1}(\omega)},
			\bV^{j,\mathbf{s}}_{n-1}\bpr{\Theta^{n-1}_{M_{n-1}}(\omega),	Z_{ \mathcal{Y}^{n,m,j}_{ N - n + 1 }(\omega), k }^{ n, m,j }(\omega)}
			}}
			}^2,
		\end{split}
		\end{equation}
	for every
		$n\in\{1,2,\ldots,N\}$,
		$m\in\N$,
		$\mathbf{s}\in\R^{\varsigma}$
	let
		$\Phi^{n,m,\mathbf{s}}\colon\R^{\mathfrak{d}}\times\Omega\to\R^{\mathfrak{d}}$ 
	satisfy 
		for every
			$\omega\in\Omega$,
			$\theta\in\{\vartheta\in\R^{\mathfrak{d}}\colon \text{$(\R^{\mathfrak d}\ni \eta\mapsto \phi^{n,m,\mathbf{s}}(\eta,\omega)\in \R)$ is differentiable at $\vartheta$}\}$
		that
		\begin{align}
			\Phi^{n,m,\mathbf{s}}(\theta,\omega) 
			= 
			(\nabla_{\theta}\phi^{n,m,\mathbf{s}})(\theta,\omega)
			,
		\end{align}
	let 
		$\S^n\colon\R^{\varsigma}\times\R^{\mathfrak{d}}\times(\R^d)^{\{0,1,\ldots,N\}\times\N}\to\R^{\varsigma}$, $n\in\{1,2,\ldots,N\}$,
		be functions,
	for every
		$n\in\{1,2,\ldots,N\}$,
		$m\in\N$
	let 
		$\psi^n_m\colon\R^{\varrho}\to\R^{\mathfrak{d}}$	and 
		$\Psi^n_m\colon\R^{\varrho}\times\R^{\mathfrak{d}}\to\R^{\varrho}$
		be functions,
	and for every
		$n\in\{1,2,\ldots,N\}$
	let
		$\bS^{n}\colon\N_0\times\Omega\to\R^{\varsigma}$ and
		$\Xi^{n}\colon\N_0\times\Omega\to\R^{\varrho}$
		be stochastic processes
		which satisfy 
			for every 
				$m\in\N$ 
			that
			\begin{equation}\label{eq:general_batch_normalization}
				\bS^{n}_{m} 
				= 
				\S^{n}\bpr{\bS^{n}_{m-1}, \Theta^{n}_{m-1},	(\Y_k^{n,m,i})_{(k,i)\in\{0,1,\ldots,N\}\times\N}}
				,
			\end{equation}
			\begin{equation}
				\label{eq:general_gradient_step}
				\Xi^n_{m} 
				=
				\Psi^n_{m}(\Xi^n_{m-1},\Phi^{n,m,\bS^n_{m}}(\Theta^n_{m-1}))
				,
				\qquad\text{and}\qquad
				\Theta^{n}_{m} 
				=
				\Theta^{n}_{m-1} - \psi^n_{m}(\Xi^n_{m})
				.
			\end{equation}
\end{algo}
In the setting of \cref{def:general_algorithm} above
we think under suitable hypotheses 
for sufficiently large
$N \in \N$,
sufficiently large
$(M_n)_{n\in \N_0} \subseteq \N$,
sufficiently large 
$(K_n)_{n\in \N} \subseteq \N$,
every
$n \in \{0, 1, \dots, N\}$, 
and every 
$x \in \D$
of 
$ \mathbb{V}^{1,\mathbb{S}_{M_n}^n}_n(\Theta^n_{M_n},x) \colon \Omega \to \R$ as a suitable approximation
\begin{equation}
\mathbb{V}^{1,\mathbb{S}_{M_n}^n}_n(\Theta^n_{M_n},x) \approx u(t_n,x)
\end{equation}
of $u(t_n,x)$ where
$
u=(u(t,x))_{(t,x)\in [0,T]\times\R^d}\in C^{1,2}([0,T]\times\R^d,\R)
$
is a function with at most  polynomially growing derivatives
which satisfies for every
$t\in (0,T]$,
$x\in \partial_\D$
that
$ \ang*{\mathbf{n}(x) ,(\nabla_x u)(t,x)} = 0$
and which satisfies for every
$t\in [0,T]$,
$x\in\D$
that
$u(0,x)=g(x)$,
$\int_\D \babs{f\bpr{t,x,\mathbf{x}, \allowbreak u(t,x), u(t,\mathbf{x})
} } \, \nu_x(\diff\mathbf{x}) < \infty$,
and
\begin{equation}
\label{eq:defPDE_general}
\begin{split}
\bpr{\tfrac{\partial}{\partial t}u}(t,x)
&=
\int_{\D} f\bpr{t,x,\mathbf{x}, u(t,x),u(t,\mathbf{x})
} \, \nu_x(\diff\mathbf{x}) \\
& \quad + \bang{\mu(x), ( \nabla_x u )( t,x ) }
+ \tfrac{ 1 }{ 2 }
\Trace\bpr{
\sigma(x) [ \sigma(x) ]^*
( \Hess_x u)( t,x )
}
\end{split}
\end{equation}
(cf.\ \eqref{eq:defPDE}).
Compared to the simplified algorithm in \cref{algo:special_case} above, the major new elements introduced in \cref{def:general_algorithm} are the following:
\begin{enumerate}[(a)]
	\item The numbers of gradient descent steps taken to compute approximations for the solution of the PDE at the times $t_n$, $n\in\{1,2,\dots,N\}$, are allowed to vary with $n$, and so are specified by a sequence $(M_n)_{n\in\N_0}\subseteq\N$ in \cref{def:general_algorithm} above.
	\item The numbers of samples used for the Monte Carlo approximation of the non-local term in the approximation for the solution of the PDE at the times $t_n$, $n\in\{1,2,\dots,\allowbreak N\}$, are allowed to vary with $n$, and so are specified by a sequence $(K_n)_{n\in\N_0}\subseteq\N$ in \cref{def:general_algorithm} above.
	\item The approximating functions $\mathbb V^{j,\mathbf s}_n$, $(j,\mathbf s,n)\in \N\times\R^\varsigma\times\{0,1,\dots,N\}$, in \cref{def:general_algorithm} above are not specified concretely in order to allow for a variety of neural network architectures. For the concrete choice of these functions employed in our numerical simulations, we refer the reader to \cref{sec:examples}.
	\item For every $m\in\{1,2,\dots,M\}$ the loss function used in the $m$th gradient descent step may be computed using a minibatch of samples instead of just one sample (cf.~\cref{eq:loss_general_case}). The sizes of these minibatches are specified by a sequence $(J_m)_{m\in\N}\subseteq\N$.
	\item Compared to \cref{algo:special_case} above, the more general form of the PDEs considered in this section (cf.~\cref{eq:defPDE_general}) requires more flexibility in the definition of the time-discrete stochastic processes $\mathcal Y^{n,m,j}\colon \{0,1,\dots,N\}\times\Omega\to\R^d$, $(n,m,j)\in\{1,2,\dots,N\}\times\N\times\N$, which are specified in \cref{def:general_algorithm} above in terms of the Brownian motions $W^{n,m,j}\colon[0,T]\times\Omega\to\R^d$, $(n,m,j)\in\{1,2,\dots,N\}\times\N\times\N$, via a function $H\colon [0,T]^2\times\R^d\times\R^d\to\R^d$ (cf.~\cref{eq:FormalXapprox}). We refer the reader to \cref{eq:Hfkpp} in \cref{subsec:fisherKPP_neumann_r} below, \cref{eq:Hcomp} in \cref{subsec:nonlocalcompPDE} below, \cref{eq:Hsinegordon} in \cref{subsec:sinegordon_nonlocal} below, \cref{eq:Hrepmut} in \cref{subsec:aniso_mutator_selector} below, and \cref{eq:Hallencahn} in \cref{subsec:allen_cahn} below for concrete choices of $H$ in the approximation of various example PDEs.
	\item For every $n\in\{1,2,\dots,N\}$, $m\in\N$ the optimization step in \cref{eq:general_gradient_step} in \cref{def:general_algorithm} above is specified generically in terms of the functions $\psi^n_m\colon \R^{\varrho}\to\R^{\mathfrak d}$ and $\Psi^n_m\colon\R^{\varrho}\times\R^{\mathfrak d}\to\R^{\varrho}$ and the random variable $\Xi^n_m\colon \Omega\to\R^{\varrho}$. This generic formulation covers a variety of SGD based optimization algorithms such as Adagrad \cite{duchi2011adaptive}, RMSprop, or Adam \cite{Kingma2014}. For example, in order to implement the Adam optimization algorithm, for every $n\in\{1,2,\dots,N\}$, $m\in\N$ the random variable $\Xi^n_m$ can be used to hold suitable first and second moment estimates (see~\cref{eq:examples_setting_moment_estimation} and \cref{eq:examples_setting_adam_grad_update} in \cref{sec:examples} below for the concrete specification of these functions implementing the Adam optimization algorithm).
	\item The processes $\mathbb{S}^n\colon \N_0\times\Omega \to \R^{\varsigma}$, $n\in\{1,2,\ldots,N\}$, and functions $\mathcal S^n\colon \R^{\varsigma}\times\R^{\mathfrak d}\times(\R^d)^{\{0,1,\dots,N\}\times\N}\to\R^\varsigma$, $n\in\{1,2,\dots,N\}$, in \cref{def:general_algorithm} above can be used to implement batch normalization; see \cite{ioffe2015batch} for details. Loosely speaking, for every $n\in\{1,2,\dots,N\}$, $m\in\N$ the random variable $\mathbb S^n_m\colon \Omega\to\R^\varsigma$ then holds mean and variance estimates of the outputs of each layer of the approximating neural networks related to the minibatches that are used as inputs to the neural networks in computing the loss function at the corresponding gradient descent step.
\end{enumerate}

\section{Multilevel Picard approximation method for non-local PDEs}
\label{sec:MLP}

In this section we introduce in \cref{frame:mlpsetting} in \cref{subsec:mlpmethod} below our extension of the full history recursive multilevel Picard approximation method for approximating solutions of non-local nonlinear PDEs with Neumann boundary conditions. The MLP method was first introduced in E et al.~\cite{Weinan2021} and Hutzenthaler et al.~\cite{Hutzenthaler2020} and later extended in a number of directions; see E et al.~\cite{E2020} and Beck et al.~\cite{Beck2020} for recent surveys. We also refer the reader to Becker et al.~\cite{becker2020numerical} and E et al.~\cite{E2019multilevel} for numerical simulations illustrating the performance of MLP methods across a range of example PDE problems.

In \cref{subsec:mlp_examples} below, we will specify five concrete examples of (non-local) nonlinear PDEs and describe how \cref{frame:mlpsetting} can be specialized to compute approximate solutions to these example PDEs. These computations will be used in \cref{sec:examples} to obtain reference values to compare the deep learning-based approximation method proposed in \cref{sec:derivation_gen} above against.

\subsection{Description of the proposed approximation method}
\label{subsec:mlpmethod}
\begingroup
\newcommand{\Index}{\mathfrak{I}}
\newcommand{\dindex}{\mathfrak{i}}
\begin{algo}[Multilevel Picard approximation method]\label{frame:mlpsetting}
	Assume 
		\cref{algo:time_discrete_reflected_processes},
	let
		$c,T\in (0,\infty)$, 
		$\Index = \bcup_{n\in \N} \Z^n$,
		$f \in C([0,T]\times D \times D \times \R \times \R,\R)$, 
		$g\in C(D,\R)$, 
		$u \in C([0,T]\times \D, \R)$, 
	assume
		$u\vert_{[0,T)\times \D}\in C^{1,2}([0,T)\times \D,\R)$, 
	let
		$\nu_x \colon \mathcal{B}( \D )\to [0,1]$,
		$x \in \D$,
		be probability measures, 
	for every
		$x \in \D$
	let
		$ \Zz^{\dindex }_{ x } \colon \Omega \to \D $,
		$\dindex \in \Index$,
		be i.i.d.\ random variables, 
	assume for every
		$ A \in \mathcal{B}( \D ) $,
		$\dindex \in \Index$
	that
		$\P( \Zz_{ x }^{ \dindex } \in A ) = \nu_x( A )$,
	let 
		$ \phi_r \colon \R \rightarrow \R$,
		$r\in [0,\infty]$, 
	satisfy 
		for every 
			$r\in [0,\infty]$, 
			$y \in \R$ 
		that
		\begin{equation}
			\phi_r(y) = \min\{r,\max\{-r,y\}\},
		\end{equation}
	let 
		$(\Omega,\mathcal{F}, \P)$ 
		be a probability space,
	let 
		$\Rr^\dindex\colon \Omega \rightarrow (0,1),$ $\dindex \in \Index$, 
		be independent $\mathcal{U}_{(0,1)}$-distributed random variables, 
	let 
		$V^\dindex\colon [0,T]\times \Omega \rightarrow [0,T]$, $\dindex \in \Index$, 
	satisfy 
		for every 
			$t\in [0,T]$, 
			$\dindex \in \Index$ 
		that 
		\begin{equation}
			V^\dindex_{t} 
			= 
			t+ (T-t)\Rr^\dindex
			,
		\end{equation} 
	let 
		$W^\dindex\colon [0,T]\times \Omega \rightarrow \R^d$, $\dindex \in \Index$, 
		be independent standard Brownian motions, 
	assume that 
		$(\Rr^\dindex)_{\dindex \in \Index}$ and $(W^\dindex)_{\dindex \in \Index}$ are independent, 
	let 
		$\mu \colon \R^{d} \to \R^{d}$ and 
		$\sigma \colon \R^{d}\to \R^{d\times d}$ 
		be globally Lipschitz continuous, 
	for every 
		$x\in \R^{d}$, 
		$\dindex \in \Index$, 
		$t\in [0,T]$ 
	let 
		$X^{x,\dindex}_{t} = (X^{x,\dindex}_{t,s})_{s\in [t,T]} \colon [t,T] \times  \Omega \rightarrow \R^d$
		be a stochastic process with continuous sample paths,
	let
		$(K_{n,l,m})_{n,l,m\in \N_0} \subseteq \N$,
	for every
		$\dindex \in \Index$, 
		$n,M \in \N_0$, 
		$r\in [0,\infty]$
	let 
		$U^\dindex_{n,M,r}\colon [0,T]\times \R^d\times \Omega \rightarrow \R^k$
	satisfy for every 
		$t\in [0,T]$, 
		$x\in \R^d$ 
	that 
	\begin{equation}
	\label{setting:MLP}
	\begin{split}
		&
		U^\dindex_{n,M,r}(t,x) 
		= 
		\Biggl[\sum_{l=0}^{n-1} \frac{(T-t)}{M^{n-l}}  
		\sum_{m=1}^{M^{n-l}} \frac{1}{K_{n,l,m}}
		\sum_{k=1}^{K_{n,l,m}}
		\bbbbr{ f \bbpr{
				V_t^{(\dindex,l,m)},
				X^{x,(\dindex,l,m)}_{t,V_t^{(\dindex,l,m)}},
				\Zz^{(\dindex,l,m,k) }_{ X^{x,(\dindex,l,m)}_{t,V_t^{(\dindex,l,m)}} },\\
			&\quad 
				\phi_{r}\bbpr{U^{(\dindex,l,m)}_{l,M,r}\bpr{V_t^{(\dindex,l,m)},X^{x,(\dindex,l,m)}_{t,V_t^{(\dindex,l,m)}}}},
				\phi_{r}\bbpr{U^{(\dindex,l,m)}_{l,M,r}\bpr{V_t^{(\dindex,l,m)},\Zz^{(\dindex,l,m,k) }_{ X^{x,(\dindex,l,m)}_{t,V_t^{(\dindex,l,m)}}}}}
			} \\
			&\quad- \mathbbm{1}_\N(l) \, f \bbpr{
				V_t^{(\dindex,l,m)},
				X^{x,(\dindex,l,m)}_{t,V_t^{(\dindex,l,m)}}, 
				\Zz^{(\dindex,l,m,k)}_{ X^{x,(\dindex,l,m)}_{t,V_t^{(\dindex,l,m)}} },
				\phi_{r}\bbpr{U^{(\dindex,l,-m)}_{\max\{l-1,0\},M,r}\bpr{V_t^{(\dindex,l,m)},X^{x,(\dindex,l,m)}_{t,V_t^{(\dindex,l,m)}}}},\\
			&\quad 
				\phi_{r}\bbpr{U^{(\dindex,l,-m)}_{\max\{l-1,0\},M,r}\bpr{V_t^{(\dindex,l,m)},\Zz^{(\dindex,l,m,k) }_{ X^{x,(\dindex,l,m)}_{t,V_t^{(\dindex,l,m)}} }}}
			}}\Biggr] 
			+  
			\frac{\1_{\N}(n)}{M^n} \bbbbbr{\sum_{m=1}^{M^n} g\bpr{X^{x,(\dindex,0,-m)}_{t,T}}  },
	\end{split}
	\end{equation}
	assume for every 
		$t \in [0,T)$,
		$x \in \partial_D$
	that
		$ \ang*{\mathbf{n}(x) ,(\nabla_x u)(t,x)} = 0$,
	and assume for every 
		$t\in [0,T)$, 
		$x \in \D$
	that 
		$\norm{u(t,x)}\leq c(1+\norm{x}^{c})$, 
		$u(T,x) = g(x)$, 
		and
		\begin{multline}\label{setting:PDE}
			\bpr{\tfrac{\partial}{\partial t}u}(t,x) 
			+\tfrac 12 \Trace \bpr{\sigma(x)[\sigma(x)]^{*}(\Hess_{x} u )(t,x)}
			+   \ang{\mu (x), (\nabla_x u) (t,x)}
			\\+ \int_D f (t,x,\mathbf{x},u(t,x),u(t,\mathbf{x})) \, \nu_x(\diff\mathbf{x})
			=
			0
			.
		\end{multline}
	\end{algo}

\subsection{Examples for the approximation method}
\label{subsec:mlp_examples}

\begin{example}[Fisher--KPP PDEs with Neumann boundary conditions]
	\label{exampleMLP:fisherkpp_neumann}
	In this example we specialize \cref{frame:mlpsetting} to the case of certain Fisher--KPP PDEs with Neumann boundary conditions (cf., e.g., Bian et al.~\cite{Bian2017} and Wang et al.~\cite{Wang2021}).
	
	Assume 
		\cref{frame:mlpsetting},
	let
		$\epsilon\in(0,\infty)$
	satisfy
		$\epsilon = \tfrac{1}{10}$,
	assume that
		$d\in\{1,2,5,10\}$,
		$D = [-\nicefrac{1}{2},\nicefrac{1}{2}]^d$, and
		$T\in\{\nicefrac{1}{5},\nicefrac{1}{2},1\}$,
	assume 
		for every 
			$n,l,m\in \N$ 
		that
			$K_{n,l,m} = 1$,
	assume for every 
		$t \in [0,T]$,
		$x,\mathbf{x} \in \D$,
		$y,\mathbf{y} \in \R$,
		$v \in \R^d$ 
	that
		$g(x)= \exp (- \tfrac{1}{4}\norm{x}^2)$,
		$\mu(x)=(0,\dots,0)$, 
		$\sigma(x) v = \epsilon v$, and
		$f(t,x,\mathbf{x},y,\mathbf{y})= y(1-y)$,
	and	assume that for every 
		$x\in \R^{d}$, 
		$\dindex \in \Index$, 
		$t\in [0,T]$, 
		$s\in [t,T]$ 
	it holds $\P$-a.s.\ that
	\begin{equation}
		X^{x,\dindex}_{t,s} 
		= 
		R\pr*{x,x + \int_{t}^{s} \mu\bpr{X^{x,\dindex}_{t,r}} \, \diff r + \int^{s}_{t} \sigma \bpr{X^{x,\dindex}_{t,r}} \, \diff W^{\dindex}_{r} }
		=
		R(x,x+\epsilon(W^\dindex_s-W^\dindex_t))
		.
	\end{equation}
	The solution 
		$u\colon[0,T]\times \D \to \R$ 
	of the PDE in \eqref{setting:PDE} then satisfies that 
		for every
			$t\in [0,T)$, $x\in\partial_\D$
		it holds that
			$\ang*{\mathbf{n}(x) ,(\nabla_x u)(t,x)} = 0$
		and that for every
			$t\in [0,T)$, $x\in\D$
		it holds that
			$u(T,x) = \exp (- \tfrac{1}{4}\norm{x}^2)$ and
			\begin{equation}
				\label{eqMLP:fisherKPP_neumann}
				\bpr{\tfrac{\partial }{\partial t}u} (t,x) 
				+
				\tfrac{\epsilon^2}{2}(\Delta_x u) (t,x) 
				+
				u(t,x)\bpr{1 - u(t,x) }
				=
				0
				.
			\end{equation}
\end{example}

\begin{example}[Non-local competition PDEs]

	\label{exampleMLP:nonlocal_comp}
	In this example we specialize \cref{frame:mlpsetting} to the case of certain non-local competition PDEs (cf., e.g., Doebeli \& Ispolatov \cite{Doebeli2010}, Berestycki et al.~\cite{Berestycki2009b}, Perthame \& Génieys \cite{Perthame2007}, and Génieys et al.~\cite{Genieys2006a}).

	Assume 
		\cref{frame:mlpsetting},
	let
		$\mathfrak s,\epsilon\in(0,\infty)$
	satisfy 
		$\mathfrak{s} = \epsilon = \tfrac{1}{10}$,
	assume that
		$d\in\{1,2,5,10\}$,
		$D = \R^d$, and
		$T\in\{\nicefrac{1}{5},\nicefrac{1}{2},1\}$,
	assume for every 
		$n,l,m\in \N$ 
	that
		$K_{n,l,m} = 10$,
	assume for every
		$x \in \R^d$,
		$A \in \mathcal{B}(\R^d)$
	that
		$\nu_x(A) = \pi^{-\nicefrac{d}{2}}\mathfrak{s}^{-d}\int_A \exp\pr*{-\mathfrak{s}^{-2}\norm{x - \mathbf{x}}^2}\,\diff\mathbf{x}$,
	assume for every 
		$t \in [0,T]$,
		$v,x,\mathbf{x} \in \R^d$,
		$y,\mathbf{y} \in \R$
	that
		$g(x)= \exp (- \tfrac{1}{4}\norm{x}^2)$,
		$\mu(x)=(0,\dots,0)$,
		$\sigma(x) v = \epsilon v$, and
		$f(t,x,\mathbf{x},y,\mathbf{y})= y(1 -\mathbf{y}\pi^{\nicefrac{d}{2}}\mathfrak{s}^d)$,
	and	assume that for every 
		$x\in \R^{d}$, 
		$\dindex \in \Index$, 
		$t\in [0,T]$, 
		$s\in [t,T]$ 
	it holds $\P$-a.s.\ that
	\begin{equation}
		X^{x,\dindex}_{t,s} 
		= 
		x + \int_{t}^{s} \mu\bpr{X^{x,\dindex}_{t,r}} \, \diff r + \int^{s}_{t} \sigma \bpr{X^{x,\dindex}_{t,r}} \, \diff W^{\dindex}_{r}
		=
		x+\epsilon(W^\dindex_s-W^\dindex_t)
		.
	\end{equation}
	The solution 
		$u\colon[0,T]\times \R^d \to \R$ 
		of the PDE in \eqref{setting:PDE} then satisfies that
	for every
		$t\in [0,T)$, 
		$x\in\R^d$
	it holds that
		$u(T,x) = \exp (-\tfrac{1}{4}\norm{x}^2)$ and
		\begin{equation}
			\label{eqMLP:nonlocalcompet}
			\bpr{\tfrac{\partial }{\partial t}u} (t,x) 
			+
			\tfrac{\epsilon^2}{2}(\Delta_x u) (t,x) 
			+ 
			u(t,x)\pr*{1 - \int_{\R^d} u(t,\mathbf{x})\,\exp\bpr{-\tfrac{\norm{x-\mathbf{x}}^2}{\mathfrak{s}^2}} \, \diff\mathbf{x} }
			=
			0
			.
	\end{equation}
\end{example}

\begin{example}[Non-local sine-Gordon PDEs]

	\label{exampleMLP:sinegordon_nonlocal}
	In this example we specialize \cref{frame:mlpsetting} to the case of certain non-local sine-Gordon type PDEs (cf., e.g., Hairer \& Shen \cite{Hairer2016}, Barone et al.~\cite{Barone1971}, and Coleman \cite{Coleman1994}).

	Assume 
		\cref{frame:mlpsetting},
	let
		$\mathfrak s,\epsilon\in(0,\infty)$
	satisfy 
		$\mathfrak{s} = \epsilon = \tfrac{1}{10}$,
	assume that
		$d\in\{1,2,5,10\}$,
		$D = \R^d$, 
		and	$T\in\{\nicefrac{1}{5},\nicefrac{1}{2},1\}$,
	assume for every 
		$n,l,m\in \N$ 
	that
		$K_{n,l,m} = 10$,
	assume for every
		$x \in \R^d$,
		$A \in \mathcal{B}(\R^d)$
	that
		$\nu_x(A) = \pi^{-\nicefrac{d}{2}}\mathfrak{s}^{-d}\int_A \exp\pr*{-\mathfrak{s}^{-2}\norm{x - \mathbf{x}}^2}\,\diff\mathbf{x}$,
	assume for every 
		$t \in [0,T]$,
		$v,x,\mathbf{x} \in \R^d$,
		$y,\mathbf{y} \in \R$
	that
		$g(x)= \exp (- \tfrac{1}{4}\norm{x}^2)$,
		$\mu(x)=(0,\dots,0)$,
		$\sigma(x) v = \epsilon v$, and
		$f(t,x,\mathbf{x},y,\mathbf{y})= \sin(y) -\mathbf{y}\pi^{\nicefrac{d}{2}}\mathfrak{s}^d$,
	and	assume that for every 
		$x\in \R^{d}$, 
		$\dindex \in \Index$, 
		$t\in [0,T]$, 
		$s\in [t,T]$ 
	it holds $\P$-a.s.\ that
	\begin{equation}
		X^{x,\dindex}_{t,s} 
		= 
		x + \int_{t}^{s} \mu\bpr{X^{x,\dindex}_{t,r}} \, \diff r + \int^{s}_{t} \sigma \bpr{X^{x,\dindex}_{t,r}} \, \diff W^{\dindex}_{r}
		=
		x+\epsilon(W^\dindex_s-W^\dindex_t)
		.
	\end{equation}
	The solution 
		$u\colon[0,T]\times \R^d \to \R$ 
		of the PDE in \eqref{setting:PDE} then satisfies that for every
			$t\in [0,T)$, 
			$x\in\R^d$
		it holds that
			$u(T,x) = \exp (-\tfrac{1}{4}\norm{x}^2)$ and
		\begin{equation}
			\label{eqMLP:sinegordon_nonlocal}
			\bpr{\tfrac{\partial}{\partial t}u}(t,x)
			+
			\tfrac{\epsilon^2}{2}(\Delta_x u)(t,x) 
			+ 
			\sin ( u(t,x) ) 
			- 
			\int_{\R^d} u(t,\mathbf{x})\, \exp\bpr{-\tfrac{\norm{x-\mathbf{x}}^2}{\mathfrak{s}^2}}\,\diff \mathbf{x} 
			=
			0
			.
		\end{equation}
\end{example}

\begin{example}[Replicator-mutator PDEs]
	\label{exampleMLP:hamel}
	In this example we specialize \cref{frame:mlpsetting} to the case of certain $d$-dimensional replicator-mutator PDEs (cf., e.g., Hamel et al.~\cite{Hamel2020}). 

	Assume 
		\cref{frame:mlpsetting}, 
	let
		$\mathfrak{m}_1, \mathfrak{m}_2, \dots,\mathfrak{m}_d, \mathfrak{s}_1, \mathfrak{s}_2, \dots, \mathfrak{s}_d,\mathfrak{u}_1, \mathfrak{u}_2,\dots,\mathfrak{u}_d \in \R$
	satisfy 
		for every
			$k \in \{1,2,\dots,d\}$
		that
			$\mathfrak{m}_k = \tfrac{1}{10}$,
			$\mathfrak{s}_k = \tfrac{1}{20}$,	and
			$\mathfrak{u}_k = 0$,
	assume that
		$d\in\{1,2,5,10\}$,
		$\D = \R^{d}$, and
		$T\in\{\nicefrac{1}{5},\nicefrac{1}{2},1\}$,
	assume for every 
		$n,l,m\in \N$
	that 
		$K_{n,l,m} = 10$,
	let
		$a \colon \R^d \to \R$
	satisfy 
		for every
			$x \in \R^d$
		that
			$a(x) = -\frac{1}{2}\norm{x}^2$,
	assume for every
		$x \in \R^d$,
		$A \in \mathcal{B}(\R^d)$
	that
		$\nu_x(A) = \int_{A\cap[-\nicefrac12,\nicefrac12]^d} \diff {\mathbf x}$,
	assume for every
		$t \in [0,T]$,
		$v= (v_1,\dots,v_d),\,x = (x_1,\dots,x_d)\in \R^d$,
		$\mathbf{x} \in \R^d$,
		$y,\mathbf{y} \in \R$
	that
		$g(x)= (2\pi)^{-\nicefrac{d}{2}} \bbr{\prod_{ i = 1 }^d \abs{\mathfrak{s}_i}^{-\nicefrac{1}{2}}} \exp \bpr{-\sum_{i = 1}^d \, \frac{(x_i - \mathfrak{u}_i )^2}{2\mathfrak{s}_i}}$,
		$\mu(x)=(0,\dots,0)$,
		$\sigma(x)v=(\mathfrak{m}_1 v_1, \dots, \mathfrak{m}_d v_d)$, and
	\begin{equation}
		f(t,x,\mathbf{x},y,\mathbf{y}) 
		=
		y \pr*{a(x) -  \mathbf{y}  a({\mathbf x}) }
		,
	\end{equation}
	and assume that 
		for every 
			$x\in \R^{d}$, 
			$\dindex \in \Index$, 
			$t\in [0,T]$, 
			$s\in [t,T]$ 
		it holds $\P$-a.s.\ that
		\begin{equation}
			X^{x,\dindex}_{t,s} 
			= 
			x + \int_{t}^{s} \mu\bpr{X^{x,\dindex}_{t,r}} \, \diff r + \int^{s}_{t} \sigma \bpr{X^{x,\dindex}_{t,r}} \, \diff W^{\dindex}_{r} 
			=
			x+\sigma(0)(W^\dindex_s-W^\dindex_t)
			.
		\end{equation}
	The solution 
		$u\colon[0,T]\times \R^d \to \R$ 
		of the PDE in \eqref{setting:PDE} then satisfies that 
			for every
				$t\in [0,T)$,
				$x = (x_1,\dots,x_d) \in \R^d$
			it holds that
				$u(T,x)= (2\pi)^{-\nicefrac{d}{2}}\bbr{ \prod_{ i = 1 }^d \abs{\mathfrak{s}_i}^{-\nicefrac{1}{2}}} \exp \bpr{-\sum_{i = 1}^d \, \frac{(x_i - \mathfrak{u}_i )^2}{2\mathfrak{s}_i}}$ and
			\begin{equation}
				\label{eqMLP:aniso_mutator_selector}
				\bpr{\tfrac{\partial }{\partial t}u} (t,x) 
				+
				u(t,x)\pr*{a(x) - \int_{[-\nicefrac12,\nicefrac12]^d} u(t,\mathbf{x})\, a(\mathbf{x}) \, \diff\mathbf{x} } 
				+ \smallsuml_{i=1}^{d} \tfrac{1}{2} \abs{\mathfrak{m}_i}^2 \bpr{\tfrac{\partial^2 }{\partial x_i^2} u} (t,x) 
				=
				0
				.
		\end{equation}
\end{example}

\begin{example}[Allen--Cahn PDEs with conservation of mass]
	\label{exampleMLP:allen_cahn}
	In this example we specialize \cref{frame:mlpsetting} to the case of certain Allen--Cahn PDEs with cubic nonlinearity, conservation of mass, and no-flux boundary conditions (cf., e.g., Rubinstein \& Sternberg \cite{RUBINSTEIN1992}).
	
	Assume 
		\cref{frame:mlpsetting},
	let
		$\epsilon\in(0,\infty)$
	satisfy 
		$\epsilon = \tfrac{1}{10}$,
	assume that
		$d\in\{1,2,5,10\}$,
		$D = [-\nicefrac12, \nicefrac12]^d$, and
		$T\in\{\nicefrac{1}{5},\nicefrac{1}{2},1\}$,
	assume for every 
		$n,l,m\in \N$ 
	that
		$K_{n,l,m} = 10$,
	assume for every
		$x \in D$,
		$A \in \mathcal{B}(D)$
	that
		$\nu_x(A) = \int_A \diff\mathbf{x}$,
	assume for every 
		$t \in [0,T]$,
		$x,\mathbf{x} \in \D$,
		$y,\mathbf{y} \in \R$,
		$v \in \R^d$ 
	that
		$g(x)= \exp (- \tfrac{1}{4}\norm{x}^2)$,
		$\mu(x)=(0,\dots,0)$,
		$\sigma(x) v = \epsilon v$, and
		$f(t,x,{\mathbf x},y,{\mathbf y})= y - y^3 - (\mathbf{y} - {\mathbf y}^3)$,
	and	assume that 
		for every 
			$x\in \R^{d}$, 
			$\dindex \in \Index$, 
			$t\in [0,T]$, 
			$s\in [t,T]$ 
		it holds $\P$-a.s.\ that
		\begin{equation}
			X^{x,\dindex}_{t,s} 
			= 
			R\pr*{x,x + \int_{t}^{s} \mu\bpr{X^{x,\dindex}_{t,r}} \, \diff r + \int^{s}_{t} \sigma \bpr{X^{x,\dindex}_{t,r}} \, \diff W^{\dindex}_{r} }
			=
			R(x,x+\epsilon(W^\dindex_s-W^\dindex_t))
			.
		\end{equation}
	The solution 
		$u\colon[0,T]\times \D \to \R$ 
	of the PDE in \eqref{setting:PDE} then satisfies that 
		for every
			$t\in [0,T)$, 
			$x\in\partial_\D$
		it holds that
			$\ang{\mathbf{n}(x) ,(\nabla_x u)(t,x)} = 0$
		and that for every
			$t\in [0,T)$, $x\in\D$
		it holds that
			$u(T,x) = \exp (- \tfrac{1}{4}\norm{x}^2)$ and			
		\begin{equation}
			\label{eqMLP:allen_cahn}
			\bpr{\tfrac{\partial}{\partial t}u}(t,x)
			+
			\tfrac{\epsilon^2}{2} (\Delta_x u)(t,x) 
			+ 
			u(t,x) - [u(t,x)]^3 
			- 
			\int_{[-\nicefrac12,\nicefrac12]^d} u(t,\mathbf{x}) - [u(t,\mathbf{x})]^3\, \diff \mathbf{x} 
			=
			0
			.
		\end{equation}
\end{example}

\section{Numerical simulations}
\label{sec:examples}
In this section we illustrate the performance of the machine learning-based approximation method proposed in \cref{def:general_algorithm} in \cref{subsec:algo-Full-gen} above by means of numerical simulations for five concrete (non-local) nonlinear PDEs; see \cref{subsec:fisherKPP_neumann_r,subsec:nonlocalcompPDE,subsec:sinegordon_nonlocal,subsec:aniso_mutator_selector,subsec:allen_cahn} below. In each of these numerical simulations we employ the general machine learning-based approximation method proposed in \cref{def:general_algorithm} with certain $4$-layer neural networks and using the Adam optimizer (cf.\ \eqref{eq:examples_setting_moment_estimation} and \eqref{eq:examples_setting_adam_grad_update} in \cref{frame:adam} below and Kingma \& Ba~\cite{Kingma2014}).

More precisely, in each of the numerical simulations in \cref{subsec:fisherKPP_neumann_r,subsec:nonlocalcompPDE,subsec:sinegordon_nonlocal,subsec:aniso_mutator_selector,subsec:allen_cahn}
the functions 
$\bV^{j,\mathbf s}_n\colon \R^{\mathfrak d}\times\R^d\to\R$
with $ n \in \{ 1, 2, \dots, N\} $,
$ j \in \{ 1, 2, \dots, 8000 \} $,
$ \mathbf{s} \in \R^{\varsigma}$
are implemented as $N$
fully-connected feedforward neural networks.
These neural networks consist of
$ 4 $ layers (corresponding to 3 affine linear transformations in the neural networks) where
the input layer is $d$-dimensional (with $ d $ neurons on the input layer), where
the two hidden layers are both $(d+50)$-dimensional (with $d+50$ neurons on each of the two hidden layers), and where the output layer is $1$-dimensional (with 1 neuron on the output layer).
We refer to \cref{fig:nn} for a graphical illustration of the neural network architecture used in the numerical simulations in \cref{subsec:fisherKPP_neumann_r,subsec:nonlocalcompPDE,subsec:sinegordon_nonlocal,subsec:aniso_mutator_selector,subsec:allen_cahn}.

As activation functions just in front of the two hidden layers we employ, in \cref{subsec:fisherKPP_neumann_r,subsec:nonlocalcompPDE,subsec:sinegordon_nonlocal,subsec:aniso_mutator_selector} below, multidimensional versions of the hyperbolic tangent function
\begin{equation}
	\R \ni x \mapsto (e^x + e^{-x})^{-1} (e^{x} - e^{-x}) \in \R,
\end{equation}
and we employ, in \cref{subsec:allen_cahn} below, multidimensional versions of the ReLU function
\begin{equation}
	\R\ni x\mapsto \max\{x,0\}\in\R.
\end{equation}
In addition, in \cref{subsec:fisherKPP_neumann_r,subsec:nonlocalcompPDE,subsec:aniso_mutator_selector} we use the square function 
$
\R \ni x \mapsto x^2 \in \R
$
as activation function just in front of the output layer and in \cref{subsec:sinegordon_nonlocal,subsec:allen_cahn} we use the identity function 
$
\R \ni x \mapsto x \in \R
$
as activation function just in front of the output layer.
Furthermore, we employ Xavier initialization to initialize all neural network parameters; see Glorot \& Bengio \cite{glorot2010} for details. We did not employ batch normalization in our simulations.

Each of the numerical experiments presented below was performed with the {\sc Julia} library {\sc HighDimPDE.jl} on a NVIDIA TITAN RTX GPU with 1350 MHz core clock and 24 GB GDDR6 memory with 7000 MHz clock rate where the underlying system consisted of an AMD EPYC 7742 64-core CPU with 2TB memory running {\sc Julia} 1.7.2 on Ubuntu 20.04.3.
We refer to \cref{sec:sourcecodes} below for the employed {\sc Julia} source codes.

\begin{figure}
	\centering
	\begin{tikzpicture}[x=4.3cm,y=1.2cm]
		\readlist\Nnod{3,7,7,1} %
		\readlist\Nstr{d,d+50,} %
		\readlist\Cstr{x,h^{(\prev)},{\mathbb{V}^{1,0}_n(\theta,x)}} %
		\def\yshift{0.55} %
		
		\foreachitem \N \in \Nnod{
		\def\lay{\Ncnt} %
		\pgfmathsetmacro\prev{int(\Ncnt-1)} %
		\foreach \i [evaluate={\c=int(\i==\N); 
					\y=\lay>0?\N/2-\i-\c*\yshift:\N/2-\i;
					\x=\lay; 
					\n=\nstyle;
					\index=(\i<\N?int(\i):"\Nstr[\n]");}] in {1,...,\N}{ %
			\node[node \n] (N\lay-\i) at (\x,\y) {$\strut\Cstr[\n]_{\index}$};
			
			\ifnumcomp{\lay}{>}{1}{ %
			\foreach \j in {1,...,\Nnod[\prev]}{ %
				\draw[white,line width=1.2,shorten >=1] (N\prev-\j) -- (N\lay-\i);
				\draw[connect] (N\prev-\j) -- (N\lay-\i);
			}
			}{
			}
			
		}
		\ifnum \lay> 0
			\ifnum \lay<4
				\path (N\lay-\N) --++ (0,1+\yshift) node[midway,scale=1.6] {$\vdots$}; %
			\fi
		\fi
		}
		
		\node[above=.1,align=center,mydarkgreen] at (N1-1.90) {Input layer\\[-0.2em](1st layer)};
		\node[above=.1,align=center,mydarkblue] at (N2-1.90) {1st hidden layer\\[-0.2em](2nd layer)};
		\node[above=.1,align=center,mydarkblue] at (N3-1.90) {2nd hidden layer\\[-0.2em](3rd layer)};
		\node[above=.1,align=center,mydarkred] at (N\Nnodlen-1.90) {Output layer\\[-0.2em](4th layer)};
	\end{tikzpicture}
	\caption{Graphical illustration of the neural network architecture used in the numerical simulations. In \cref{subsec:fisherKPP_neumann_r,subsec:nonlocalcompPDE,subsec:sinegordon_nonlocal,subsec:aniso_mutator_selector,subsec:allen_cahn} we employ neural networks with $4$ layers
	(corresponding to $3$ affine linear transformations in the neural networks) 
	with $d$ neurons on the input layer (corresponding to a $d$-dimensional input layer), 
	with $d + 50$ neurons on the 1st hidden layer (corresponding to a $(d+50)$-dimensional 1st hidden layer),
	with $d + 50$ neurons on the 2nd hidden layer (corresponding to a $(d+50)$-dimensional 2nd hidden layer), and 
	with 1 neuron on the output layer (corresponding to a $1$-dimensional output layer) in the numerical simulations. 
	}
	\label{fig:nn}
\end{figure}
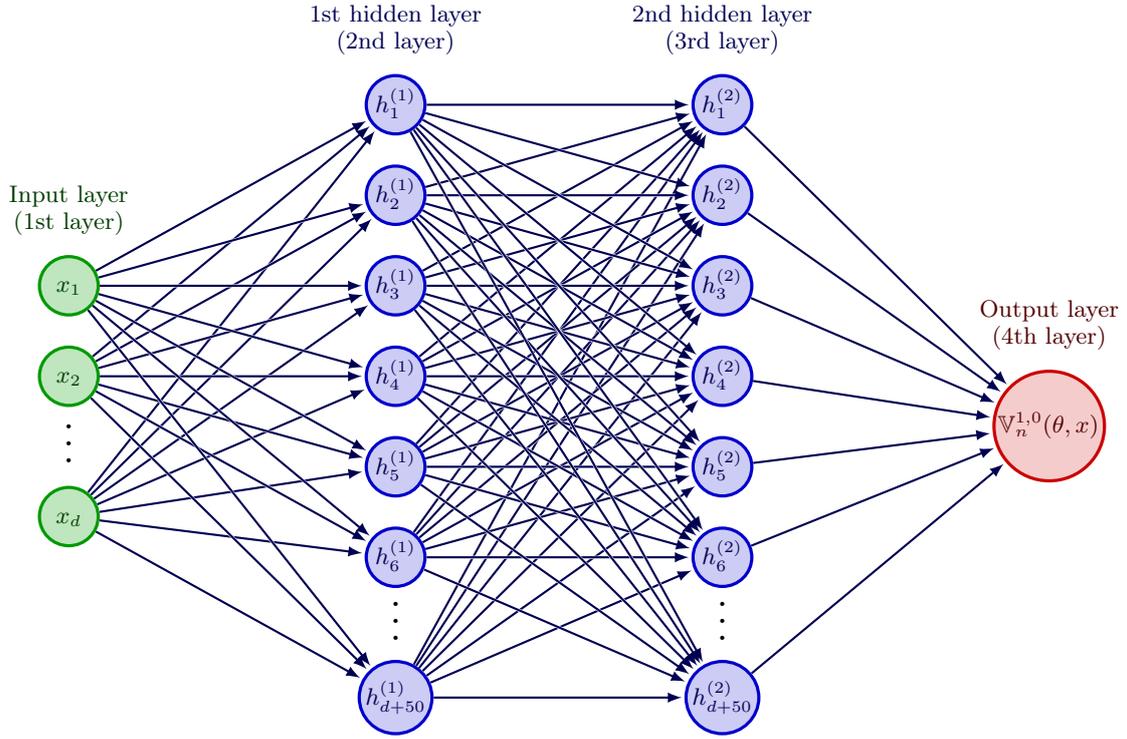

\begin{algo}
	\label{frame:adam}
	Assume 
		\cref{def:general_algorithm}, 
	assume
		$\mathfrak{d}=(d+50)(d+1)+ (d+50)(d+51)+(d+51)$,
	let
		$\varepsilon,\beta_1,\beta_2\in\R$,
		$(\gamma_m)_{m\in\N}\subseteq (0,\infty)$
	satisfy
		$\varepsilon=10^{-8}$,
		$\beta_1 = \tfrac{9}{10}$,
		and $\beta_2 = \tfrac{999}{1000}$,
	let
		$g\colon D \to \R$,
		$\mu\colon D\to\R^d$, and
		$\sigma\colon D\to\R^{d\times d}$
		be continuous,
	let %
		$u=(u(t,x))_{(t,x) \in [0,T]\times \D} \in C^{1,2}([0,T]\times \D,\R)$
		have at most polynomially growing partial derivatives, 
		assume 
			for every
				$t\in (0,T]$,
				$x\in \partial_\D$
			that
				$ \ang*{\mathbf{n}(x) ,(\nabla_x u)(t,x)} = 0$,
	assume 
		for every
			$t\in [0,T]$,
			$x\in\D$,
			$j \in \N$,
			$\mathbf{s} \in \R^\varsigma$
		that
			$u(0,x)=g(x) = \mathbb V^{ j, \mathbf{s} }_0( \theta, x )$,
			$\int_\D \babs{f\bpr{t,x,\mathbf{x}, u(t,x),u(t,\mathbf{x})
			} } \, \nu_x(\diff \mathbf{x}) < \infty$, and
		\begin{equation}
		\begin{split}
		\label{eq:PDE-Examples}
			\bpr{\tfrac{\partial}{\partial t}u}(t,x)
			&=
			\tfrac{ 1 }{ 2 } \Trace\bpr{\sigma(x) [ \sigma(x) ]^*( \Hess_x u)( t,x )}
			+ \bang{ \mu(x), ( \nabla_x u )( t,x ) }
			\\& \quad 
			+ \int_{\D} f\bpr{t,x,\mathbf{x}, u(t,x),u(t,\mathbf{x})
			} \, \nu_x(\diff\mathbf{x}) 
			,
		\end{split}
		\end{equation}
	assume 
		for every
			$m\in\N$,
			$i\in\{0,1,\ldots,N\} $ 
		that
			$ J_m = 8000$,
			$ t_i = \tfrac{iT}{N} $, and
			$ \varrho = 2 \mathfrak{d} $,
	and assume 
		for every
			$n \in \{1,2,\dots,N\}$,
			$m\in\N$,
			$x=(x_1, \ldots, x_{\mathfrak{d}})$, $y=(y_1, \ldots, y_{\mathfrak{d}})$, $\eta = ( \eta_1, \ldots , \eta_{\mathfrak{d}} )\in\R^{\mathfrak{d}}$
		that
		\begin{align}\label{eq:examples_setting_moment_estimation}
			\Xi^n_0(x,y,\eta) = 0,\quad 
			\Psi^n_m ( x , y , \eta )
			=
			\bpr{\beta_1 x + (1-\beta_1) \eta, \beta_2 y + (1-\beta_2) ((\eta_1)^2,\ldots,(\eta_{\mathfrak{d}})^2)},
		\end{align}
		and
		\begin{align}\label{eq:examples_setting_adam_grad_update}
			\psi^n_m ( x,y ) 
			=
			\bbbpr{
			\bbbr{
			\sqrt{\tfrac{\abs{y_1}}{1-(\beta_2)^m}} + \varepsilon
			}^{-1}
			\frac{\gamma_m x_{1}}{1-(\beta_1)^m},
			\ldots,
			\bbbr{
			\sqrt{\tfrac{\abs{y_{\mathfrak{d}}}}{1-(\beta_2)^m}} + \varepsilon
			}^{-1}
			\frac{\gamma_m x_{\mathfrak{d}}}{1-(\beta_1)^m}
			}
			.
	\end{align}
\end{algo}

\subsection{Fisher--KPP PDEs with Neumann boundary conditions}
\label{subsec:fisherKPP_neumann_r}
In this subsection we use the machine learning-based approximation method in \cref{frame:adam} to approximately calculate the solutions of certain Fisher--KPP PDEs with Neumann boundary conditions (cf., e.g., Bian et al.~\cite{Bian2017} and Wang et al.~\cite{Wang2021}).

Assume 
	\cref{frame:adam},
let
	$\epsilon\in(0,\infty)$
satisfy
	$\epsilon = \tfrac{1}{10}$,
assume that
	$d\in\{1,2,5,10\}$,
	$D = [-\nicefrac{1}{2},\nicefrac{1}{2}]^d$,
	$T\in\{\nicefrac{1}{5},\nicefrac{1}{2},1\}$,
	$N=10$,
	$K_1 = K_2 = \ldots = K_N= 1$, and
	$M_1 = M_2 = \ldots = M_N = 500$,
assume 
	for every 
		$n,m,j\in \N$, $\omega \in \Omega$ 
	that 
		$\xi^{n,m,j}(\omega)=(0,\dots,0)$,
assume 
	for every 
		$m \in \N$
	that
		$\gamma_m = 10^{-2}$,
and assume 
	for every 
		$s,t \in [0,T]$,
		$x,\mathbf{x} \in \D$,
		$y,\mathbf{y} \in \R$,
		$v
		\in \R^d$ 
	that
		$g(x)=  \exp (-\tfrac{1}{4}\norm{x}^2)$,
		$\mu(x)=(0,\dots,0)$,
		$\sigma(x) v =\epsilon v$, 
		$f(t,x,\mathbf{x},y,\mathbf{y}
		)= y(1-y)$, and
		\begin{equation}
			\label{eq:Hfkpp}
			H(t,s,x,v) 
			= 
			R(x,x+\mu(x)(t-s)+\sigma(x)v)
			=
			R(x,x+\epsilon v)
		\end{equation} 
		(cf.\ \eqref{Y-algo-spez} and \eqref{eq:FormalXapprox}).
The solution 
	$u\colon[0,T]\times \D \to \R$ 
	of the PDE in \eqref{eq:PDE-Examples} then satisfies that 
		for every
			$t\in (0,T]$, 
			$x\in\partial_\D$
		it holds that
			$\ang*{\mathbf{n}(x) ,(\nabla_x u)(t,x)} = 0$
	and that 
		for every
			$t\in [0,T]$, $x\in\D$
		it holds that
			$u(0,x) =  \exp (-\tfrac{1}{4}\norm{x}^2)$ and
		\begin{equation}
			\label{eq:fisherKPP_neumann_r}
    		\bpr{\tfrac{\partial }{\partial t}u} (t,x) 
			=
			\tfrac{\epsilon^2}{2}(\Delta_x u) (t,x) + u(t,x)\bpr{1 - u(t,x) }
			.
		\end{equation}
In \eqref{eq:fisherKPP_neumann_r} the function $u\colon [0,T]\times D\to\R$ models the proportion of a particular type of alleles in a biological population spatially structured over $D$.
For every $t\in[0,T]$, $x\in\R^d$ the number $u(t,x) \in \R$ describes the proportion of individuals with a particular type of alleles located at position $x = (x_1, \dots, x_d) \in \R^d$ at time $t \in [0,T]$.
In \cref{table:fisherKPP_neumann_r} 
we use the machine learning-based approximation method
in \cref{frame:adam}
to approximately calculate
the mean of %
$
\bV^{1,0}_N(\Theta^N_{M_N},\allowbreak (0,\ldots,0))
$,
the standard deviation of %
$
\bV^{1,0}_N(\Theta^N_{M_N},\allowbreak (0,\ldots,0))
$,
the relative $ L^1 $-ap\-prox\-i\-ma\-tion error associated to %
$
\bV^{1,0}_N(\Theta^N_{M_N},\allowbreak (0,\ldots,0))
$,
the uncorrected sample standard deviation of the approximation error associated to %
$
\bV^{1,0}_N(\Theta^N_{M_N},\allowbreak (0,\ldots,0))
$,
and the average runtime in seconds needed for calculating one realization of $
\bV^{1,0}_N(\Theta^N_{M_N},\allowbreak (0,\ldots,0))
$
based on $5$ independent realizations (5 independent runs).
The reference value, which is used as an approximation for the unknown value $u(T,(0,\ldots,0))$ of the exact solution of \eqref{eq:fisherKPP_neumann_r}, has been calculated via the MLP approximation method for non-local nonlinear PDEs in \cref{frame:mlpsetting} (cf.~\cref{exampleMLP:fisherkpp_neumann} and Beck et al.~\cite[Remark~3.3]{Beck2017a}).

\begin{table}
	\begin{center}
		\resizebox{\textwidth}{!}{
		\begin{approxtabular}
			$1$ & $\nicefrac{1}{5}$ & $10$ & $0.9995902$ & $0.0000107$ & $0.9996057$ & $0.0000155$ & $0.0000107$ & $24.887$\\
			$2$ & $\nicefrac{1}{5}$ & $10$ & $0.9991759$ & $0.0000191$ & $0.9991887$ & $0.0000186$ & $0.0000116$ & $26.175$\\
			$5$ & $\nicefrac{1}{5}$ & $10$ & $0.9979572$ & $0.0000388$ & $0.9979693$ & $0.0000303$ & $0.0000235$ & $27.312$\\
			$10$ & $\nicefrac{1}{5}$ & $10$ & $0.9959224$ & $0.0000341$ & $0.9959337$ & $0.0000275$ & $0.0000196$ & $28.972$\\\hline
			$1$ & $\nicefrac{1}{2}$ & $10$ & $0.9992463$ & $0.0000341$ & $0.9992572$ & $0.0000237$ & $0.0000248$ & $26.631$\\
			$2$ & $\nicefrac{1}{2}$ & $10$ & $0.9984982$ & $0.0000287$ & $0.9985442$ & $0.0000460$ & $0.0000287$ & $27.007$\\
			$5$ & $\nicefrac{1}{2}$ & $10$ & $0.9962227$ & $0.0000330$ & $0.9962314$ & $0.0000306$ & $0.0000041$ & $27.632$\\
			$10$ & $\nicefrac{1}{2}$ & $10$ & $0.9925257$ & $0.0001663$ & $0.9921744$ & $0.0003541$ & $0.0001676$ & $28.743$\\\hline
			$1$ & $1$ & $10$ & $0.9991423$ & $0.0000331$ & $0.9989768$ & $0.0001657$ & $0.0000332$ & $26.601$\\
			$2$ & $1$ & $10$ & $0.9982349$ & $0.0000782$ & $0.9982498$ & $0.0000605$ & $0.0000430$ & $26.965$\\
			$5$ & $1$ & $10$ & $0.9956516$ & $0.0000853$ & $0.9957053$ & $0.0000839$ & $0.0000466$ & $27.428$\\
			$10$ & $1$ & $10$ & $0.9912297$ & $0.0001072$ & $0.9904936$ & $0.0007431$ & $0.0001083$ & $28.521$\\
		\end{approxtabular}
		}
	\end{center}
	\caption{Numerical simulations for the approximation method in \cref{def:general_algorithm} in the case of the Fisher--KPP PDEs with Neumann boundary conditions in \eqref{eq:fisherKPP_neumann_r} in \cref{subsec:fisherKPP_neumann_r}.
	\label{table:fisherKPP_neumann_r}}
\end{table}

\subsection{Non-local competition PDEs}
\label{subsec:nonlocalcompPDE}
In this subsection we use the machine learning-based approximation method in \cref{frame:adam}
to approximately calculate the solutions of certain non-local competition PDEs (cf., e.g., Doebeli \& Ispolatov \cite{Doebeli2010}, Berestycki et al.~\cite{Berestycki2009b}, Perthame \& Génieys \cite{Perthame2007}, and Génieys et al.~\cite{Genieys2006a}).

Assume 
	\cref{frame:adam}, 
let
	$\mathfrak s,\epsilon\in(0,\infty)$
satisfy
	$\mathfrak{s} = \epsilon =\tfrac{1}{10}$,
assume that
	$d\in\{1,2,5,10\}$,
	$\D = \R^d$,
	$T\in\{\nicefrac{1}{5},\nicefrac{1}{2},1\}$,
	$N=10$,
	$K_1 = K_2 = \ldots = K_N= 5$, and
	$M_1 = M_2 = \ldots = M_N = 500$,
assume 
	for every 
		$n,m,j\in \N$, $\omega \in \Omega$ 
	that 
		$\xi^{n,m,j}(\omega)=(0,\dots,0)$,
assume
	for every 
		$m \in \N$
	that
		$\gamma_m = 10^{-2}$,
and assume 
	for every 
		$s,t \in [0,T]$, 
		$v,x,\mathbf{x}%
		\in \R^d$, 
		$y,\mathbf{y} \in \R$,
		$A \in \mathcal{B}(\R^d)$
	that
		$\nu_x(A) = \pi^{-\nicefrac{d}{2}}\mathfrak{s}^{-d}\int_A \exp\pr*{-\mathfrak{s}^{-2}\norm{x - \mathbf{x}}^2}\,\diff\mathbf{x}$,
		$g(x)= \exp (-\tfrac{1}{4}\norm{x}^2)$,
		$\mu(x)=(0,\dots,0)$,
		$\sigma(x) v=\epsilon v$, 
		$f(t,x,\mathbf{x},y,\mathbf{y}%
		)=  y(1 - \mathbf{y} \mathfrak{s}^d\pi^{\nicefrac{d}{2}})$, and
	\begin{equation}
		\label{eq:Hcomp}
		H(t,s,x,v)
		=
		x + \mu(x)(t-s)+ \sigma(x)v
		=
		x+\epsilon v
	\end{equation}
	(cf.\ \eqref{Y-algo-spez} and \eqref{eq:FormalXapprox}).
The solution 
	$u\colon[0,T]\times \R^d \to \R$ 
	of the PDE in \eqref{eq:PDE-Examples} then satisfies that 
		for every
			$t\in [0,T]$, $x\in\R^d$ 
		it holds that 
			$u(0,x)=\exp (-\tfrac{1}{4}\norm{x}^2)$ and
		\begin{equation}
			\label{eq:nonlocalcompPDE}
 			\bpr{\tfrac{\partial}{\partial t}u}(t,x)
 			=
 			\tfrac{\epsilon^2}{2}(\Delta_x u)(t,x) + u(t,x) \pr*{1 - \int_{\R^d} u(t,\mathbf{x}) \exp \bpr{-\tfrac{\norm{x-\mathbf{x}}^2}{\mathfrak{s}^2} } \, \diff \mathbf{x} }
			.
		\end{equation}
In \eqref{eq:nonlocalcompPDE} the function $u\colon[0,T]\times\R^d\to\R$ models the evolution of a population characterized by a set of $d$ biological traits under the combined effects of selection, competition and mutation. 
For every $t\in[0,T]$, $x\in\R^d$ the number $u(t,x) \in \R$ describes the number of individuals with traits $x = (x_1, \dots, x_d) \in \R^d$ at time $t \in [0,T]$.
In \cref{table:nonlocalcompPDE}
we use the machine learning-based approximation method
in \cref{frame:adam}
to approximately calculate
the mean of %
$
\bV^{1,0}_N(\Theta^N_{M_N},(0,\ldots,0))
$,
the standard deviation of %
$
\bV^{1,0}_N(\Theta^N_{M_N},(0,\ldots,0))
$,
the relative $ L^1 $-approximation error associated to %
$
\bV^{1,0}_N(\Theta^N_{M_N},(0,\ldots,0))
$,
the uncorrected sample standard deviation of the approximation error associated to %
$
\bV^{1,0}_N(\Theta^N_{M_N},(0,\ldots,0))
$,
and the average runtime in seconds needed for calculating one realization of $
\bV^{1,0}_N(\Theta^N_{M_N},(0,\ldots,0))
$
based on $5$ independent realizations (5 independent runs).
The reference value, which is used as an approximation for the unknown value $u(T,(0,\ldots,0))$ of the exact solution of \eqref{eq:nonlocalcompPDE},  has been calculated via the MLP approximation method for non-local nonlinear PDEs in \cref{frame:mlpsetting} (cf.~\cref{exampleMLP:nonlocal_comp} and Beck et al.~\cite[Remark~3.3]{Beck2017a}).

\begin{table}
	\begin{center}
		\resizebox{\textwidth}{!}{\begin{approxtabular}
			$1$ & $\nicefrac{1}{5}$ & $5$ & $1.1748404$ & $0.0006512$ & $1.1735975$ & $0.0010591$ & $0.0005549$ & $20.571$\\
			$2$ & $\nicefrac{1}{5}$ & $5$ & $1.2114236$ & $0.0008700$ & $1.2096305$ & $0.0014823$ & $0.0007193$ & $25.042$\\
			$5$ & $\nicefrac{1}{5}$ & $5$ & $1.2186650$ & $0.0007070$ & $1.2159038$ & $0.0022709$ & $0.0005814$ & $54.644$\\
			$10$ & $\nicefrac{1}{5}$ & $5$ & $1.2153864$ & $0.0007789$ & $1.2128666$ & $0.0020776$ & $0.0006422$ & $74.331$\\\hline
			$1$ & $\nicefrac{1}{2}$ & $5$ & $1.4755801$ & $0.0032738$ & $1.4694976$ & $0.0041392$ & $0.0022278$ & $20.182$\\
			$2$ & $\nicefrac{1}{2}$ & $5$ & $1.6112576$ & $0.0110426$ & $1.5948898$ & $0.0103067$ & $0.0068414$ & $25.178$\\
			$5$ & $\nicefrac{1}{2}$ & $5$ & $1.6433913$ & $0.0067468$ & $1.6186897$ & $0.0152602$ & $0.0041681$ & $53.618$\\
			$10$ & $\nicefrac{1}{2}$ & $5$ & $1.6323552$ & $0.0053956$ & $1.6090688$ & $0.0144720$ & $0.0033532$ & $73.648$\\\hline
			$1$ & $1$ & $5$ & $2.0795628$ & $0.0223341$ & $2.0493301$ & $0.0147525$ & $0.0108982$ & $19.836$\\
			$2$ & $1$ & $5$ & $2.5651031$ & $0.0513671$ & $2.4683060$ & $0.0392160$ & $0.0208107$ & $24.700$\\
			$5$ & $1$ & $5$ & $2.6977694$ & $0.0381160$ & $2.5606137$ & $0.0535636$ & $0.0148855$ & $52.343$\\
			$10$ & $1$ & $5$ & $2.6490054$ & $0.0155291$ & $2.5299994$ & $0.0470380$ & $0.0061380$ & $73.186$\\
			\hline
		\end{approxtabular}}
	\end{center}
	\caption{Numerical simulations for the approximation method in \cref{def:general_algorithm} in the case of the non-local competition PDEs in \eqref{eq:nonlocalcompPDE} in \cref{subsec:nonlocalcompPDE}.
	\label{table:nonlocalcompPDE}
	}
\end{table}

\subsection{Non-local sine-Gordon type PDEs}
\label{subsec:sinegordon_nonlocal}
In this subsection we use the machine learning-based approximation method in \cref{frame:adam}
to approximately calculate the solutions of non-local sine-Gordon type PDEs 
(cf., e.g., Hairer \& Shen \cite{Hairer2016}, Barone et al.~\cite{Barone1971}, and Coleman \cite{Coleman1994}).

Assume 
	\cref{frame:adam},
let
	$\mathfrak s,\epsilon\in(0,\infty)$
satisfy
	$\mathfrak{s} = \epsilon =\tfrac{1}{10}$,
assume that
	$d\in\{1,2,5,10\}$,
	$\D = \R^d$,
	$T\in\{\nicefrac{1}{5},\nicefrac{1}{2},1\}$,
	$N=10$,
	$K_1 = K_2 = \ldots = K_N= 5$, and
	$M_1 = M_2 = \ldots = M_N = 500$,
assume 
	for every 
		$n,m,j\in \N$, 
		$\omega \in \Omega$
	that 
		$\xi^{n,m,j}(\omega)=(0,\dots,0)$,
assume 
	for every 
		$m \in \N$
	that
		$\gamma_m = 10^{-3}$,
and assume 
	for every 
		$s,t \in [0,T]$, 
		$v,x,\mathbf{x}%
		\in \R^d$, 
		$y,\mathbf{y} \in \R$,
		$A \in \mathcal{B}(\R^d)$
	that
		$\nu_x(A) = \pi^{-\nicefrac{d}{2}}\mathfrak{s}^{-d}\int_A \exp\pr*{-\mathfrak{s}^{-2}\norm{x - \mathbf{x}}^2}\,\diff\mathbf{x}$,
		$g(x)= \exp (-\tfrac{1}{4}\norm{x}^2)$,
		$\mu(x)=(0,\dots,0)$,
		$\sigma(x) v = \epsilon v$, 
		$f(t,x,\mathbf{x},y,\mathbf{y}%
		)=  \sin(y) - \mathbf{y} \pi^{\nicefrac{d}{2}}\mathfrak{s}^d$, 
		and
	\begin{equation}
		\label{eq:Hsinegordon}
		H(t,s,x,v)
		=
		x + \mu(x)(t-s)+ \sigma(x)v
		=
		x+\epsilon v
	\end{equation}
	(cf.\ \eqref{Y-algo-spez} and \eqref{eq:FormalXapprox}).
The solution 
	$u\colon[0,T]\times \R^d \to \R$ 
	of the PDE in \eqref{eq:PDE-Examples} then satisfies that 
		for every
			$t\in [0,T]$, 
			$x\in\R^d$ 
		it holds that 
			$u(0,x)=\exp (-\tfrac{1}{4}\norm{x}^2)$ and
		\begin{equation}
			\label{eq:sinegordon_nonlocal}
			\bpr{\tfrac{\partial}{\partial t}u}(t,x)
			=
			\tfrac{\epsilon^2}{2}(\Delta_x u)(t,x) + \sin ( u(t,x) ) - \int_{\R^d} u(t,\mathbf{x})\, \exp\bpr{-\tfrac{\norm{x-\mathbf{x}}^2}{\mathfrak{s}^2}}\,\diff\mathbf{x} 
			.
		\end{equation}
In \cref{table:sinegordon_nonlocal} 
we use the machine learning-based approximation method
in \cref{frame:adam} 
to approximately calculate
the mean of %
$
\bV^{1,0}_N(\Theta^N_{M_N},\allowbreak (0,\ldots,0))
$,
the standard deviation of %
$
\bV^{1,0}_N(\Theta^N_{M_N},\allowbreak (0,\ldots,0))
$,
the relative $ L^1 $-approximation error associated to %
$
\bV^{1,0}_N(\Theta^N_{M_N},\allowbreak (0,\ldots,\allowbreak 0))
$,
the uncorrected sample standard deviation of the approximation error associated to %
$
\bV^{1,0}_N(\Theta^N_{M_N},\allowbreak (0,\ldots,0))
$,
and the average runtime in seconds needed for calculating one realization of $
\bV^{1,0}_N(\Theta^N_{M_N},\allowbreak (0,\ldots,0))
$
based on $5$ independent realizations (5 independent runs).
The reference value, which is used as an approximation for the unknown value $u(T,(0,\ldots,0))$ of the exact solution of \eqref{eq:sinegordon_nonlocal},  has been calculated via the MLP approximation method for non-local nonlinear PDEs in \cref{frame:mlpsetting} (cf.~\cref{exampleMLP:sinegordon_nonlocal} and Beck et al.~\cite[Remark~3.3]{Beck2017a}).

\begin{table}
	\begin{center}
		\resizebox{\textwidth}{!}{\begin{approxtabular}
			$1$ & $\nicefrac{1}{5}$ & $10$ & $1.1363013$ & $0.0000101$ & $1.1366512$ & $0.0003079$ & $0.0000089$ & $23.635$\\
			$2$ & $\nicefrac{1}{5}$ & $10$ & $1.1678476$ & $0.0000118$ & $1.1685004$ & $0.0005586$ & $0.0000101$ & $24.788$\\
			$5$ & $\nicefrac{1}{5}$ & $10$ & $1.1731812$ & $0.0000087$ & $1.1740671$ & $0.0007546$ & $0.0000074$ & $24.233$\\
			$10$ & $\nicefrac{1}{5}$ & $10$ & $1.1704700$ & $0.0000063$ & $1.1715686$ & $0.0009377$ & $0.0000054$ & $24.767$\\
			\hline
			$1$ & $\nicefrac{1}{2}$ & $10$ & $1.3514235$ & $0.0000152$ & $1.3529022$ & $0.0010930$ & $0.0000112$ & $22.622$\\
			$2$ & $\nicefrac{1}{2}$ & $10$ & $1.4393708$ & $0.0000245$ & $1.4423641$ & $0.0020753$ & $0.0000170$ & $23.419$\\
			$5$ & $\nicefrac{1}{2}$ & $10$ & $1.4546282$ & $0.0000816$ & $1.4598476$ & $0.0035754$ & $0.0000559$ & $23.739$\\
			$10$ & $\nicefrac{1}{2}$ & $10$ & $1.4473282$ & $0.0000739$ & $1.4503958$ & $0.0021150$ & $0.0000510$ & $24.222$\\
			\hline
			$1$ & $1$ & $10$ & $1.7114614$ & $0.0000309$ & $1.7136091$ & $0.0012533$ & $0.0000180$ & $22.067$\\
			$2$ & $1$ & $10$ & $1.9019763$ & $0.0000288$ & $1.9062322$ & $0.0022326$ & $0.0000151$ & $22.707$\\
			$5$ & $1$ & $10$ & $1.9364921$ & $0.0000602$ & $1.9411610$ & $0.0024052$ & $0.0000310$ & $22.899$\\
			$10$ & $1$ & $10$ & $1.9223347$ & $0.0001494$ & $1.9272222$ & $0.0025360$ & $0.0000775$ & $23.719$\\
			\hline
		\end{approxtabular}}
	\end{center}
	\caption{Numerical simulations for the approximation method in \cref{def:general_algorithm} in the case of the non-local sine-Gordon PDEs in \eqref{eq:sinegordon_nonlocal} in \cref{subsec:sinegordon_nonlocal}.
	\label{table:sinegordon_nonlocal}}
\end{table}

\subsection{Replicator-mutator PDEs}
\label{subsec:aniso_mutator_selector}
In this subsection we use the machine learning-based approximation method in \cref{frame:adam} to approximately calculate the solutions of certain replicator-mutator PDEs describing the dynamics of a phenotype distribution under the combined effects of selection and mutation (cf., e.g., Hamel et al.~\cite{Hamel2020}).

Assume 
	\cref{frame:adam}, 
let
	$\mathcal D\subseteq \R^d$,
	$\mathfrak{m}_1,\mathfrak{m}_2,\dots,\mathfrak{m}_d,\mathfrak{s}_1,\mathfrak{s}_2,\dots,\mathfrak{s}_d,\mathfrak{u}_1, \mathfrak{u}_2, \dots,\mathfrak{u}_d,\mathfrak t \in \R$
satisfy 
	for every
		$k \in \{1,2,\dots,d\}$
	that
		$\mathfrak{m}_k = \tfrac{1}{10}$,
		$\mathfrak{s}_k = \tfrac{1}{20}$,
		$\mathfrak{u}_k = 0$,
		and
		$\mathfrak t=\tfrac1{50}$,
assume that
	$d\in\{1,2,5,10\}$,
	$\D = \R^d$,
	$T\in\{\nicefrac1{10},\nicefrac{1}{5},\nicefrac{1}{2}\}$,
	$N=10$,
	$K_1 = K_2 = \ldots = K_N= 5$,
let
	$a\in C(\R^d,\R)$,
	$\delta\in C(\R^d,(0,\infty))$
satisfy 
	for every
		$x \in \R^d$
	that
		$a(x) = -\frac{1}{2}\norm{x}^2$,
and assume 
	for every 
		$s,t \in [0,T]$,
		$v= (v_1,\dots,v_d),\,x = (x_1,\dots,x_d)\in \R^d$, 
		$\mathbf{x} \in \R^d$,
		$y,\mathbf{y} \in \R$,
		$A\in\Borel(\R^d)$
	that
		$\nu_x(A)=\int_{A\cap \mathcal D}\delta(\mathbf x)\,\diff\mathbf{x}$,
		$g(x)= (2\pi)^{-\nicefrac{d}{2}}\bbr{ \prod_{ i = 1 }^d \abs{\mathfrak{s}_i}^{-\nicefrac{1}{2}}} \exp \bpr{-\sum_{i = 1}^d  \frac{(x_i - \mathfrak{u}_i )^2}{2\mathfrak{s}_i}}$,
		$\mu(x)=(0,\dots,0)$,
		$\sigma(x)v=(\mathfrak{m}_1 v_1, \dots, \mathfrak{m}_d v_d)$, 
		$f(t,x,\mathbf{x},y,\mathbf{y}
		) = y(a(x)-\mathbf y a(\mathbf x)[\delta(\mathbf x)]^{-1})$,
		and
	\begin{equation}
		\label{eq:Hrepmut}
		H(t,s,x,v)
		=
		x + \mu(x)(t-s)+ \sigma(x)v
		=
		x+(\mathfrak m_1v_1,\dots,\mathfrak m_dv_d)
	\end{equation}
	(cf.\ \eqref{Y-algo-spez} and \eqref{eq:FormalXapprox}).
The solution 
	$u\colon[0,T]\times \R^d \to \R$ 
	of the PDE in \eqref{eq:PDE-Examples} then satisfies that 
		for every
			$t\in [0,T]$,
			$x = (x_1, \dots, x_d) \in \R^d$
		it holds that
			\begin{equation}
				u(0,x)= (2\pi)^{-\nicefrac{d}{2}}\bbbbr{ \smallprodl_{ i = 1 }^d \abs{\mathfrak{s}_i}^{-\nicefrac{1}{2}}} \exp \bbpr{-\smallsuml_{i = 1}^d  \tfrac{(x_i - \mathfrak{u}_i )^2}{2\mathfrak{s}_i}}
			\end{equation} 
			and
\begin{equation}
	\bpr{\tfrac{\partial }{\partial t}u} (t,x) 
	=
	u(t,x)\pr*{a(x) - \int_{\mathcal D} u(t,\mathbf{x})\, a(\mathbf{x}) \, \diff\mathbf{x} } 
	+  
	\smallsuml_{i=1}^{d} \tfrac{1}{2} \abs{\mathfrak{m}_i}^2 \bpr{\tfrac{\partial^2 }{\partial x_i^2} u} (t,x)
	.
   \label{eq:aniso_mutator_selector}
\end{equation}
In \eqref{eq:aniso_mutator_selector} the function $u\colon[0,T]\times\R^d\to\R$ models the evolution of the phenotype distribution of a population composed of a set of $d$ biological traits under the combined effects of selection and mutation. 
For every $t\in[0,T]$, $x\in\R^d$ the number $u(t,x) \in \R$ describes the number of individuals with traits $x = (x_1, \dots, x_d) \in \R^d$ at time $t \in [0,T]$.
The function $a$ models a quadratic Malthusian fitness function.

In \cref{table:aniso_mutator_selector} 
we use the machine learning-based method
in \cref{frame:adam} to approximately solve the PDE in
\cref{eq:aniso_mutator_selector} above in the case
$\mathcal D=\R^d$.
More precisely, we 
assume for every
	$n,m,j\in\N$
that
	$\xi^{n,m,j}=0$,
	$\gamma_m = \nicefrac{1}{100}$,
	$M_n=1000$
and we assume for every
	$\mathbf x\in\R^d$
that
	$\delta(\mathbf x)=(2\pi)^{-\nicefrac d2}\mathfrak t^{-d}\exp\bpr{-\tfrac{\norm{\mathbf x}^2}{2\mathfrak t^2}}$
to approximately calculate
the mean of %
$
\bV^{1,0}_N(\Theta^N_{M_N},\allowbreak (0,\ldots,0))
$,
the standard deviation of %
$
\bV^{1,0}_N(\Theta^N_{M_N},\allowbreak (0,\ldots,0))
$,
the relative $ L^1 $-approximation error associated to %
$
\bV^{1,0}_N(\Theta^N_{M_N},\allowbreak (0,\ldots,\allowbreak 0))
$,
the uncorrected sample standard deviation of the approximation error associated to %
$
\bV^{1,0}_N(\Theta^N_{M_N},\allowbreak (0,\ldots,0))
$,
and the average runtime in seconds needed for calculating one realization of $
\bV^{1,0}_N(\Theta^N_{M_N},\allowbreak (0,\ldots,0))
$
based on $5$ independent realizations (5 independent runs).
The value $u(T,(0,\ldots,0))$ of the exact solution of \eqref{eq:aniso_mutator_selector} has been calculated by means of \cref{lem:aniso_mutator_selector} below.

In \cref{fig:rep_mut} we use the machine learning-based method in \cref{frame:adam} to approximate the solution $u\colon[0,T]\times\R^d\to\R$ of the PDE in \cref{eq:aniso_mutator_selector} above with $d=5$, $T=\nicefrac12$, and $\mathcal D=\R^d$. The right-hand side of \cref{fig:rep_mut} shows a plot of $[-\nicefrac{1}{4}, \nicefrac{1}{4}] \ni x \mapsto u(t, (x, 0,  \dots,0)) \in \R$ for $t\in\{0,0.05,0.1,0.15\}$ where $u$ is the exact solution of the PDE in \cref{eq:aniso_mutator_selector} with $d=5$, $T=\nicefrac12$, and $\mathcal D=\R^d$ computed via \cref{eq:lem_aniso} in \cref{lem:aniso_mutator_selector} below. The left-hand side of \cref{fig:rep_mut} shows a plot of $[-\nicefrac{1}{4}, \nicefrac{1}{4}] \ni x \mapsto \bV_n^{1,0}(\Theta^n_{M_n}(\omega),(x, 0, \dots,0)) \in \R$ for $n\in\{0,1,2,3\}$ and one realization $\omega\in\Omega$ where the functions $\R^d\ni x\mapsto \bV_n^{1,0}(\Theta^n_{M_n}(\omega),x)\in\R$ for $n\in\{0,1,2,3\}$, $\omega\in\Omega$ were computed via \cref{frame:adam} as an approximation of the solution of the PDE in \cref{eq:aniso_mutator_selector} with $d=5$, $T=\nicefrac12$, and $\mathcal D=[-\nicefrac 12,\nicefrac 12]^d$. For the approximation, we take
$M_1=M_2=\ldots=M_N=2000$,
$\gamma_1=\gamma_2=\dots=\gamma_{2000}=\nicefrac{1}{200}$,
and
$\delta=\1_{\R^d}$
and we take
$\xi^{n,m,j}\colon\Omega\to\R^d$, $n, m, j \in \N$,
to be independent $\mathcal{U}_{[-\nicefrac{1}{2},\nicefrac{1}{2}]^d}$-distributed random variables.
Note that the solution of the PDE in \cref{eq:aniso_mutator_selector} in the case $\mathcal D=[-R,R]^d$ with $R\in(0,\infty)$ sufficiently large is a good approximation of the solution $u\colon[0,T]\times \R^d\to\R$ of the PDE in \cref{eq:aniso_mutator_selector} in the case $\mathcal D=\R^d$ since we have that for all $t\in[0,T]$ the value $u(t,x)$ of the solution $u$ of the PDE in \cref{eq:aniso_mutator_selector} in the case $\mathcal D=\R^d$ quickly tends to $0$ as $\norm x$ tends to $\infty$.

\begin{table}
	\begin{center}
		\resizebox{\textwidth}{!}{\begin{approxtabular}
			$1$ & $\nicefrac{1}{10}$ & $10$ & \multicolumn{1}{r|}{$1.7650547$} & $0.0048907$ & \multicolumn{1}{r|}{$1.7709574$} & $0.0033330$ & $0.0027616$ & $43.949$\\
			$2$ & $\nicefrac{1}{10}$ & $10$ & \multicolumn{1}{r|}{$3.1210874$} & $0.0015513$ & \multicolumn{1}{r|}{$3.1362901$} & $0.0048474$ & $0.0004946$ & $45.002$\\
			$5$ & $\nicefrac{1}{10}$ & $10$ & \multicolumn{1}{r|}{$17.1948978$} & $0.0160821$ & \multicolumn{1}{r|}{$17.4196954$} & $0.0129048$ & $0.0009232$ & $45.934$\\
			$10$ & $\nicefrac{1}{10}$ & $10$ & \multicolumn{1}{r|}{$295.8776489$} & $0.0572639$ & \multicolumn{1}{r|}{$303.4457874$} & $0.0249407$ & $0.0001887$ & $47.750$\\
			\hline
			$1$ & $\nicefrac{1}{5}$ & $10$ & \multicolumn{1}{r|}{$1.7499938$} & $0.0005580$ & \multicolumn{1}{r|}{$1.7582066$} & $0.0046711$ & $0.0003174$ & $43.129$\\
			$2$ & $\nicefrac{1}{5}$ & $10$ & \multicolumn{1}{r|}{$3.0621917$} & $0.0027811$ & \multicolumn{1}{r|}{$3.0912904$} & $0.0094131$ & $0.0008996$ & $44.443$\\
			$5$ & $\nicefrac{1}{5}$ & $10$ & \multicolumn{1}{r|}{$16.3846066$} & $0.0139748$ & \multicolumn{1}{r|}{$16.8015567$} & $0.0248162$ & $0.0008318$ & $45.019$\\
			$10$ & $\nicefrac{1}{5}$ & $10$ & \multicolumn{1}{r|}{$268.2944397$} & $0.0623432$ & \multicolumn{1}{r|}{$282.2923073$} & $0.0495864$ & $0.0002208$ & $45.612$\\
			\hline
			$1$ & $\nicefrac{1}{2}$ & $10$ & \multicolumn{1}{r|}{$1.7018557$} & $0.0060157$ & \multicolumn{1}{r|}{$1.7222757$} & $0.0118564$ & $0.0034929$ & $42.092$\\
			$2$ & $\nicefrac{1}{2}$ & $10$ & \multicolumn{1}{r|}{$2.8911286$} & $0.0027431$ & \multicolumn{1}{r|}{$2.9662336$} & $0.0253200$ & $0.0009248$ & $42.657$\\
			$5$ & $\nicefrac{1}{2}$ & $10$ & \multicolumn{1}{r|}{$14.2520916$} & $0.1356645$ & \multicolumn{1}{r|}{$15.1535149$} & $0.0594861$ & $0.0089527$ & $43.338$\\
			$10$ & $\nicefrac{1}{2}$ & $10$ & \multicolumn{1}{r|}{$201.6446228$} & $0.3009756$ & \multicolumn{1}{r|}{$229.6290127$} & $0.1218678$ & $0.0013107$ & $44.190$\\
			\hline
		\end{approxtabular}}
	\end{center}
	\caption{Numerical simulations for the approximation method in \cref{def:general_algorithm} in the case of the replicator-mutator PDEs in \eqref{eq:aniso_mutator_selector} in \cref{subsec:aniso_mutator_selector}
	where we assume for every
		$n,m,j\in\N$
	that
		$\mathcal D=\R^d$,
		$\xi^{n,m,j}=0$,
		$\gamma_m = \nicefrac{1}{100}$,
		and
		$M_n=1000$
	and where we assume for every
	  $\mathbf x\in\R^d$
	that
	  $\delta(\mathbf x)=(2\pi)^{-\nicefrac d2}\mathfrak t^{-d}\exp\bpr{-\tfrac{\norm{\mathbf x}^2}{2\mathfrak t^2}}$.
	\label{table:aniso_mutator_selector}
	}
\end{table}

\begin{figure}
	\centering
	\includegraphics[width=0.8\textwidth]{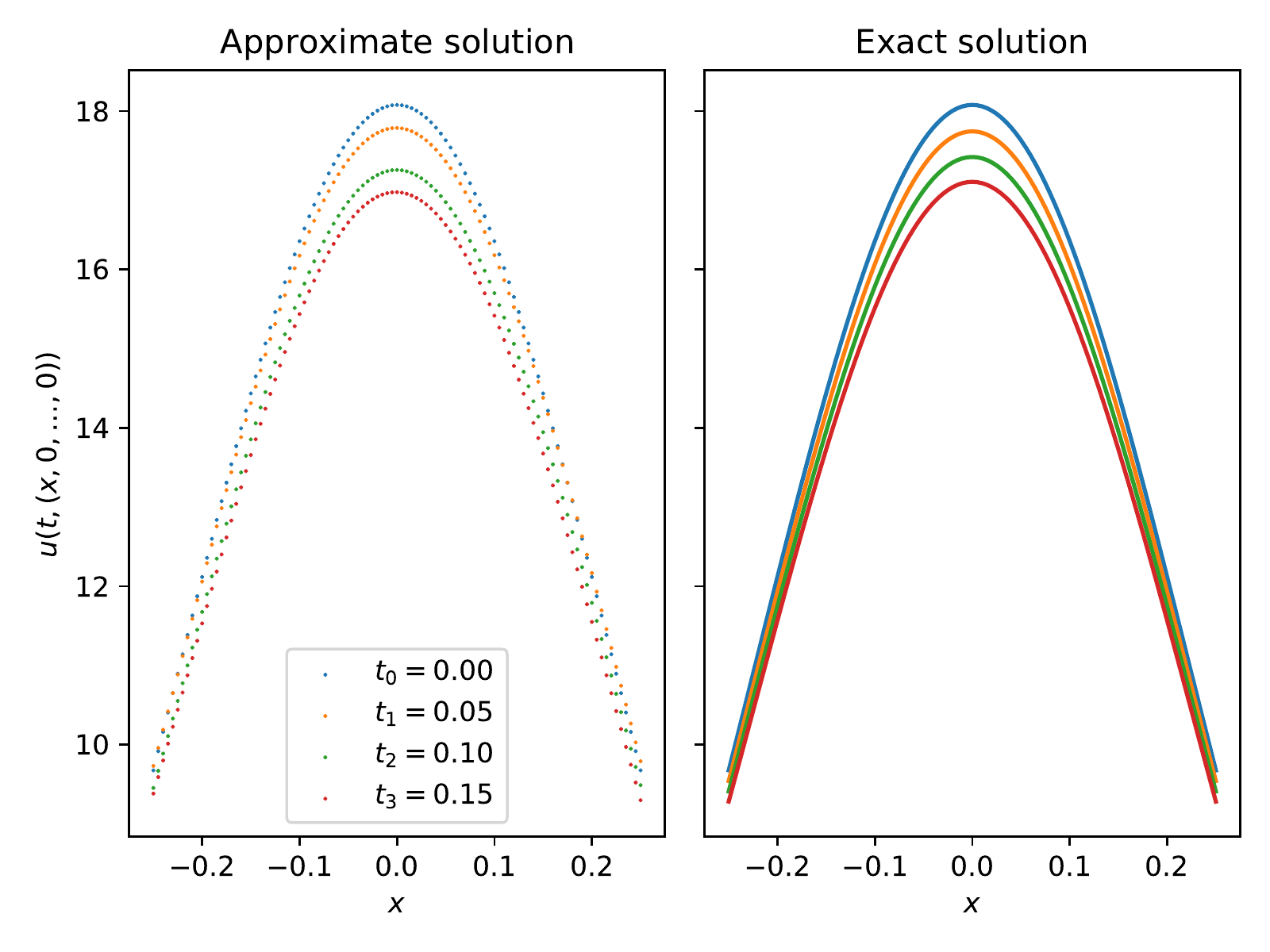}
	\caption{Plot of a machine learning-based approximation of the solution of the replicator-mutator PDE in \cref{eq:aniso_mutator_selector} in the case $d=5$, $T=\nicefrac12$, and $\mathcal D=\R^d$.
	The left-hand side shows a plot of $[-\nicefrac{1}{4}, \nicefrac{1}{4}] \ni x \mapsto \bV_n^{1,0}(\Theta^n_{M_n}(\omega),(x, 0, \dots,0)) \in \R$ for $n\in\{0,1,2,3\}$ and one realization $\omega\in\Omega$ where the functions $\R^d\ni x\mapsto \bV_n^{1,0}(\Theta^n_{M_n}(\omega),x)\in\R$ for $n\in\{0,1,2,3\}$, $\omega\in\Omega$ were computed via \cref{frame:adam} as an approximation of the solution of the PDE in \cref{eq:aniso_mutator_selector} with $d=5$, $T=\nicefrac12$, and $\mathcal D=[-\nicefrac 12,\nicefrac 12]^d$ where we take
		$M_1=M_2=\ldots=M_N=2000$,
		$\gamma_1=\gamma_2=\dots=\gamma_{2000}=\nicefrac{1}{200}$,
		and
		$\delta=\1_{\R^d}$
		and where we take
		$\xi^{n,m,j}\colon\Omega\to\R^d$, $n, m, j \in \N$,
		to be independent $\mathcal{U}_{[-\nicefrac{1}{2},\nicefrac{1}{2}]^d}$-distributed random variables.
	The right-hand side of \cref{fig:rep_mut} shows a plot of $[-\nicefrac{1}{4}, \nicefrac{1}{4}] \ni x \mapsto u(t, (x, 0,  \dots,0)) \in \R$ for $t\in\{0,0.05,0.1,0.15\}$ where $u$ is the exact solution of the PDE in \cref{eq:aniso_mutator_selector} with $d=5$, $T=\nicefrac12$, and $\mathcal D=\R^d$.
	\label{fig:rep_mut}}
\end{figure}

\begin{lemma}\label{lem:aniso_mutator_selector}
	Let
		$d \in \N$,
		$\mathfrak{u}_1,\mathfrak{u}_2, \dots,\mathfrak{u}_d\in \R$,
		$\mathfrak{m}_1,\mathfrak{m}_2, \dots,\mathfrak{m}_d,
		\mathfrak{s}_1,\mathfrak{s}_2, \dots,\mathfrak{s}_d \in (0,\infty)$,
	let
		$a \colon \R^d \to \R$
	satisfy 
		for every
			$x \in \R^d$
		that
			$a(x) = -\frac{1}{2} \norm{x}^2$,
	for every 
		$i \in \{1,2,\dots,d\}$
		let
			$\mathfrak{S}_i \colon [0,\infty) \to (0,\infty)$ and
			$\mathfrak{U}_i \colon [0,\infty) \to \R$
			satisfy for every
				$t \in [0,\infty)$
			that
			\begin{equation}\label{lem:S_U}
				\mathfrak{S}_i(t) 
				=
				\mathfrak{m}_i \br*{ \frac{\mathfrak{m}_i \sinh(\mathfrak{m}_i  t) + \mathfrak{s}_i \cosh(\mathfrak{m}_i t)}{\mathfrak{m}_i \cosh(\mathfrak{m}_i  t) + \mathfrak{s}_i \sinh(\mathfrak{m}_i t)}}
				\quad \text{and} \quad
				\mathfrak{U}_i(t)
				=
				\frac{\mathfrak{m}_i \mathfrak{u}_i}{\mathfrak{m}_i \cosh(\mathfrak{m}_i  t) + \mathfrak{s}_i \sinh(\mathfrak{m}_i t)}
				,
			\end{equation}
	and let
		$u \colon [0,\infty) \times \R^d \to \R $
	satisfy 
		for every
			$t \in [0,\infty)$,
			$x = (x_1,\dots,x_d)\in \R^d$
		that
		\begin{equation}\label{eq:lem_aniso}
			u(t,x)
			=
			(2\pi)^{- \nicefrac{d}{2}} \br*{ \smallprodl_{ i = 1 }^d \abs{\mathfrak{S}_i(t)}^{- \nicefrac{1}{2}} }     \exp\pr*{-\smallsuml_{i = 1}^d \frac{(x_i -\mathfrak{U}_i(t) )^2}{2\mathfrak{S}_i(t)}}
			.
		\end{equation}
	Then
	\begin{enumerate}[(i)]
		\item \label{lem:it_smoothness} 
			it holds that
				$u  \in C^{1,2}([0,\infty)\times\R^d,\R)$,
		\item \label{lem:it_u0} 
			it holds for every
				$x = (x_1,\dots,x_d)\in \R^d$
			that
			\begin{equation}
				u(0,x) = (2\pi)^{-\nicefrac{d}{2}} \br*{\smallprodl_{ i = 1 }^d \abs{\mathfrak{s}_i}^{-\nicefrac{1}{2}}} \exp\pr*{-\smallsuml_{i = 1}^d \frac{(x_i - \mathfrak{u}_i )^2}{2\mathfrak{s}_i}},
			\end{equation}
			and
	    \item \label{lem:it_eq} 
			it holds for every
				$t \in [0,\infty)$,
				$x = (x_1,\dots,x_d)\in \R^d$
			that
			\begin{equation}
				\label{eq:lemma_replicator_mutator}
				\bpr{\tfrac{\partial }{\partial t}u} (t,x)
				=
				u(t,x)\pr*{a(x) - \int_{\R^d} u(t,\mathbf{x})\, a(\mathbf{x}) \, \diff\mathbf{x} } + \smallsuml_{i=1}^{d} \tfrac{1}{2} \abs{\mathfrak{m}_i}^2\bpr{\tfrac{\partial^2 }{\partial x_i^2} u} (t,x)
				.
			\end{equation}
	\end{enumerate}
\end{lemma}

\begin{proof}[Proof of \cref{lem:aniso_mutator_selector}]
	First, note that 
		the fact that 
			for every
				$i \in \{1,2,\dots,d\}$
			it holds that
				$\mathfrak{S}_i \in C^{\infty}([0,\infty),(0,\infty))$,
		the fact that 
			for every
				$i \in \{1,2,\dots,d\}$
			it holds that
				$\mathfrak{U}_i \in C^{\infty}([0,\infty),\R) $,
		and \eqref{eq:lem_aniso}
	establish 
		\cref{lem:it_smoothness}.
	Moreover, observe that 
		the fact that 
			for every
				$i \in \{1,2,\dots,d\}$
			it holds that
				$\mathfrak{S}_i(0) = \mathfrak{s}_i$,
		the fact that 
			for every
				$i \in \{1,2,\dots,d\}$
			it holds that
				$\mathfrak{U}_i(0) = \mathfrak{u}_i$,
		and	\eqref{lem:S_U}
	prove  
		\cref{lem:it_u0}.
	Next note that 
		\eqref{eq:lem_aniso} 
	ensures that 
		for every
			$t \in [0,\infty)$,
			$x = (x_1,\dots,x_d)\in \R^d$
		it holds that
		\begin{equation}\label{eq:lem_u}
			u(t,x) 
			=
			\prod_{i=1}^d \br*{ (2\pi \mathfrak{S}_i(t))^{-\nicefrac{1}{2}} \exp\pr*{ - \frac{(x_i -\mathfrak{U}_i(t) )^2}{2\mathfrak{S}_i(t)}} }
			.
		\end{equation}
		The product rule 
		hence 
	implies that 
		for every
			$t \in [0,\infty)$,
			$x = (x_1,\dots,x_d)\in \R^d$
		it holds that
		\begin{equation}
		\begin{split}
			&
			\bpr{\tfrac{\partial }{\partial t}u} (t,x)  
			\\ & = 
			\frac{\partial}{\partial t} \pr*{\prod_{i=1}^d \br*{ (2\pi   \mathfrak{S}_i(t))^{-\nicefrac{1}{2}}  \exp \pr*{ - \frac{(x_i -\mathfrak{U}_i(t) )^2}{2\mathfrak{S}_i(t)}} } }
			\\ & = 
			\sum_{i=1}^d \Biggl[  \bbbbr{ \prod\nolimits_{ j \in \{ 1,\dots,d \} \backslash\{i\}} \pr*{ (2\pi \mathfrak{S}_j(t))^{-\nicefrac{1}{2}} \exp\pr*{ - \frac{(x_j -\mathfrak{U}_j(t) )^2}{2\mathfrak{S}_j(t)}} }}\\
			&\quad \cdot \br*{ \frac{\partial}{\partial t} \pr*{ (2\pi   \mathfrak{S}_i(t))^{-\nicefrac{1}{2}}   \exp\pr*{- \frac{(x_i -\mathfrak{U}_i(t) )^2}{2\mathfrak{S}_i(t)}} }  } \Biggr]
			.
		\end{split} 
		\end{equation}
		The chain rule,
		the product rule, 
		and \eqref{eq:lem_u} 
		therefore 
	show that 
		for every
			$t \in [0,\infty)$,
			$x = (x_1,\dots,x_d)\in \R^d$
		it holds that
		\begin{align}
			&\notag
			\bpr{\tfrac{\partial }{\partial t}u} (t,x)  
			\\ & = \notag
			\sum_{i=1}^d \Biggl[  \bbbbr{ \prod\nolimits_{ j \in \{ 1,\dots,d \} \backslash\{i\}} \pr*{ (2\pi \mathfrak{S}_j(t))^{-\nicefrac{1}{2}} \exp \pr*{ - \frac{(x_j -\mathfrak{U}_j(t) )^2}{2\mathfrak{S}_j(t)}} }}\\\notag
			&\quad \cdot \biggl[\pr*{ \frac{\partial}{\partial t} \pr*{ (2\pi   \mathfrak{S}_i(t))^{-\nicefrac{1}{2}} } } \, \exp\pr*{- \frac{(x_i -\mathfrak{U}_i(t) )^2}{2\mathfrak{S}_i(t)}} \\\notag
			&\quad + (2\pi   \mathfrak{S}_i(t))^{-\nicefrac{1}{2}}  \pr*{ \frac{\partial}{\partial t}  \pr*{ - \frac{(x_i -\mathfrak{U}_i(t) )^2}{2\mathfrak{S}_i(t)} } }\,\exp\pr*{- \frac{(x_i -\mathfrak{U}_i(t) )^2}{2\mathfrak{S}_i(t)}}  \biggr] \Biggr] 
			\\ &=  \label{lem:derivatives1}
			\sum_{i=1}^d \Biggl[  \bbbbr{ \prod\nolimits_{ j \in \{ 1, \dots,d \} \backslash\{i\}} \pr*{ (2\pi \mathfrak{S}_j(t))^{-\nicefrac{1}{2}} \exp\pr*{ - \frac{(x_j -\mathfrak{U}_j(t) )^2}{2\mathfrak{S}_j(t)}} }}\\\notag
			&\quad \cdot \biggl[   - (2\pi  \mathfrak{S}_i(t))^{-\nicefrac{1}{2}} \bbbbr{\frac{\pr*{ \tfrac{\partial }{\partial t}\mathfrak{S}_i} (t)}{2\mathfrak{S}_i(t)}} \exp\pr*{- \frac{(x_i -\mathfrak{U}_i(t) )^2}{2\mathfrak{S}_i(t)}} \\\notag
			&\quad +  (2\pi  \mathfrak{S}_i(t))^{-\nicefrac{1}{2}} \biggl(  \frac{ 2\bpr{\tfrac{\partial }{\partial t}\mathfrak{U}_i}(t)(x_i-\mathfrak{U_i}(t))}{2\mathfrak{S}_i(t)} \\\notag
			&\quad  + \frac{(x_i-\mathfrak{U}_i(t))^2 \bpr{\tfrac{\partial }{\partial t}\mathfrak{S}_i}(t)}{2\abs{\mathfrak{S}_i(t)}^{2}} \biggr) \exp\pr*{- \frac{(x_i -\mathfrak{U}_i(t) )^2}{2\mathfrak{S}_i(t)}} \biggr] \Biggr] 
			\\ &= \notag
			u(t,x) \bbbbbr{ \sum_{i=1}^d\pr*{ \frac{-\bpr{\tfrac{\partial }{\partial t}\mathfrak{S}_i} (t)}{2\mathfrak{S}_i(t)} +
			\frac{ 2\mathfrak{S}_i(t) \bpr{\tfrac{\partial }{\partial t}\mathfrak{U}_i}(t)(x_i-\mathfrak{U_i}(t)) + (x_i-\mathfrak{U}_i(t))^2 \bpr{\tfrac{\partial }{\partial t}\mathfrak{S}_i}(t)}{2\abs{\mathfrak{S}_i(t)}^{2}}} }
			.
		\end{align}
	Moreover, observe that 
		\eqref{lem:S_U}, 
		the chain rule, 
		and the product rule 
	ensure that 
		for every
			$i \in \{1, \dots,d\}$,
			$t \in [0,\infty)$
		it holds that
		\begin{equation}
		\begin{aligned}
			\bpr{\tfrac{\partial }{\partial t}\mathfrak{U}_i}(t) &= \frac{\partial }{\partial t} \pr*{ \frac{\mathfrak{m}_i \mathfrak{u}_i}{\mathfrak{m}_i \cosh(\mathfrak{m}_i  t) + \mathfrak{s}_i \sinh(\mathfrak{m}_i t)} }
			\\ &= 
			- \abs{\mathfrak{m}_i}^2 \mathfrak{u}_i \br*{ \frac{\mathfrak{m}_i \sinh(\mathfrak{m}_i  t) + \mathfrak{s}_i \cosh(\mathfrak{m}_i t)}{\br*{ \mathfrak{m}_i \cosh(\mathfrak{m}_i  t) + \mathfrak{s}_i \sinh(\mathfrak{m}_i t)}^2} }
			\\ &= 
			- \mathfrak{S}_i(t)    \mathfrak{U}_i(t) 
		\end{aligned}
		\end{equation}
		and
		\begin{equation}
		\begin{aligned}
			&\bpr{\tfrac{\partial }{\partial t}\mathfrak{S}_i}(t) 
			\\&= \frac{\partial }{\partial t} \pr*{ \mathfrak{m}_i \br*{ \frac{\mathfrak{m}_i \sinh(\mathfrak{m}_i  t) + \mathfrak{s}_i \cosh(\mathfrak{m}_i t)}{\mathfrak{m}_i \cosh(\mathfrak{m}_i  t) + \mathfrak{s}_i \sinh(\mathfrak{m}_i t)}}} 
			\\ &= 
			\abs{\mathfrak{m}_i}^2 \br*{ \frac{\mathfrak{m}_i \cosh(\mathfrak{m}_i  t) + \mathfrak{s}_i \sinh(\mathfrak{m}_i t)}{\mathfrak{m}_i \cosh(\mathfrak{m}_i  t) + \mathfrak{s}_i \sinh(\mathfrak{m}_i t)}} - \abs{\mathfrak{m}_i}^2\br*{ \frac{\mathfrak{m}_i \sinh(\mathfrak{m}_i  t) + \mathfrak{s}_i \cosh(\mathfrak{m}_i t)}{ \mathfrak{m}_i \cosh(\mathfrak{m}_i  t) + \mathfrak{s}_i \sinh(\mathfrak{m}_i t)} }^2
			\\ &= 
			\abs{\mathfrak{m}_i}^2  - \abs{\mathfrak{S}_i(t)}^2
			.
		\end{aligned}
		\end{equation}
	Combining 
		this 
	with 
		\eqref{lem:derivatives1} 
	implies that for every
		$i \in \{1, 2,\dots,d\}$,
		$t \in [0,\infty)$
	it holds that
	\begin{align}
		\nonumber \bpr{\tfrac{\partial }{\partial t}u} (t,x)  
		&= 
		\frac{u(t,x)}{2} \sum_{i=1}^d\Biggl[ \frac{-\bbr{\abs{\mathfrak{m}_i}^2  - \abs{\mathfrak{S}_i(t)}^2}}{\mathfrak{S}_i(t)} \\
		\nonumber &\quad + \frac{ 2\abs{\mathfrak{S}_i(t)}^2  \mathfrak{U}_i(t) (\mathfrak{U_i}(t) - x_i) + (x_i-\mathfrak{U}_i(t))^2 (\abs{\mathfrak{m}_i}^2  - \abs{\mathfrak{S}_i(t)}^2)}{\abs{\mathfrak{S}_i(t)}^{2}}\Biggr]
		\\ &= 
		\label{lem:derivatives1*}
		\frac{u(t,x)}{2} \sum_{i=1}^d\Biggl[ \abs{\mathfrak{m}_i}^2  \pr*{ \bbbpr{\frac{x_i - \mathfrak{U}_i(t)}{\mathfrak{S}_i(t)}}^{2} - \frac{1}{\mathfrak{S}_i(t)} } \\
		\nonumber
		&\quad + \mathfrak{S}_i(t) + 2\pr*{\abs{\mathfrak{U}_i(t)}^2 - \mathfrak{U}_i(t)\, x_i} - \pr*{\abs{x_i}^2 -2 \mathfrak{U}_i(t) \, x_i + \abs{\mathfrak{U}_i(t)}^2 }  \Biggr] 
		\\ &= 
		\nonumber
		\frac{u(t,x)}{2} \sum_{i=1}^d \br*{  \abs{\mathfrak{m}_i}^2  \pr*{ \bbbpr{\frac{x_i - \mathfrak{U}_i(t)}{\mathfrak{S}_i(t)}}^{2} - \frac{1}{\mathfrak{S}_i(t)} } + \mathfrak{S}_i(t) + \abs{\mathfrak{U}_i(t)}^2 - \abs{x_i}^2 }
		.
	\end{align}
	Furthermore, note that 
		\eqref{eq:lem_u} 
		and the product rule 
	show that 
		for every
			$i \in \{1,2, \dots,d\}$,
			$t \in [0,\infty)$,
			$x = (x_1, \dots,x_d)\in \R^d$
		it holds that 
		\begin{equation}
		\begin{split}
			\bpr{ \tfrac{\partial }{\partial x_i} u}(t,x)
			&=  
			\frac{\partial }{\partial x_i}  
			\br*{\prod_{j = 1}^d\br*{(2\pi   \mathfrak{S}_j(t))^{-\nicefrac{1}{2}}   \exp\pr*{ - \frac{(x_j -\mathfrak{U}_j(t) )^2}{2\mathfrak{S}_j(t)}} } }
			\\ &= 
			\br*{ \frac{\partial }{\partial x_i} 
				\br*{(2\pi   \mathfrak{S}_i(t))^{-\nicefrac{1}{2}}   \exp\pr*{ - \frac{(x_i -\mathfrak{U}_i(t) )^2}{2\mathfrak{S}_i(t)} } } }\\
				&\quad \cdot \prod_{ j \in \{ 1,2, \dots,d \} \backslash\{i\}} \br*{ (2\pi   \mathfrak{S}_j(t))^{-\nicefrac{1}{2}}   \exp\pr*{ - \frac{(x_j -\mathfrak{U}_j(t) )^2}{2\mathfrak{S}_j(t)} } }
			\\ &= 
			- u(t,x) \bbbpr{ \frac{x_i - \mathfrak{U}_i(t)}{\mathfrak{S}_i(t)} } = u(t,x) \bbbpr{ \frac{\mathfrak{U}_i(t) - x_i }{\mathfrak{S}_i(t)} }
			.
		\end{split}
		\end{equation}
		The product rule 
		therefore 
	assures that 
		for every
			$i \in \{1,2, \dots,d\}$,
			$t \in [0,\infty)$,
			$x = (x_1, \dots,x_d)\in \R^d$
		it holds that 
		\begin{equation}
		\begin{split}
			&
			\bpr{ \tfrac{\partial^2 }{\partial x_i^2} u }(t,x) 
			= 
			\frac{\partial }{\partial x_i} \pr*{u(t,x) \bbbpr{ \frac{ \mathfrak{U}_i(t) - x_i}{\mathfrak{S}_i(t)} }}
			\\ &=
			\bpr{\tfrac{\partial}{\partial x_i} u}(t,x)\bbbpr{ \frac{\mathfrak{U}_i(t) - x_i }{\mathfrak{S}_i(t)} } - \frac{u(t,x)}{\mathfrak S_i(t)}
			= 
			u(t,x) \br*{ \bbbpr{\frac{x_i - \mathfrak{U}_i(t)}{\mathfrak{S}_i(t)}}^{2} - \frac{1}{\mathfrak{S}_i(t)} }
			.
		\end{split}
		\end{equation}
		Hence, 
	we obtain that 
		for every
			$t \in [0,\infty)$,
			$x = (x_1, \dots,x_d)\in \R^d$
		it holds that 
		\begin{equation}\label{lem:derivatives2}
		\begin{split}
			&
			\sum_{i=1}^d \br*{ \tfrac{1}{2} \abs{\mathfrak{m}_i}^2\bpr{\tfrac{\partial^2 }{\partial x_i^2} u} (t,x) } 
			= 
			\frac{u(t,x)}{2}  \sum_{i=1}^d \br*{ \abs{\mathfrak{m}_i}^2 \pr*{ \bbbpr{\frac{x_i - \mathfrak{U}_i(t)}{\mathfrak{S}_i(t)}}^{2} - \frac{1}{\mathfrak{S}_i(t)} }}
			.
		\end{split}
		\end{equation}
	Next observe that 
		\eqref{eq:lem_u} 
		and Fubini's theorem
	ensure that for every
		$t \in [0,\infty)$,
		$x = (x_1, \dots,x_d)\in \R^d$
	it holds that
	\begin{equation}
	\begin{split}
		&u(t,x)\pr*{ a(x) - \int_{\R^d} u(t,\mathbf{x})\, a(\mathbf{x}) \, \diff\mathbf{x} } 
		\\ &= u(t,x) \pr*{ - \frac{1}{2} \br*{\sum_{i=1}^d \abs{x_i}^2 } - \int_{\R^d} - \frac{1}{2} \br*{ \sum_{i=1}^d \abs{\mathbf{x}_i}^2 } u(t,\mathbf{x}) \, \diff \mathbf{x} }
		\\ &= \frac{u(t,x)}{2} \Biggl( - \br*{\sum_{i=1}^d \abs{x_i}^2 } \\
		&\quad + \sum_{i=1}^d \biggl[ \int_\R \abs{\mathbf{x}_i}^2 \, (2\pi    \mathfrak{S}_i(t))^{-\nicefrac{1}{2}}    \exp\pr*{ - \frac{(\mathbf{x}_i -\mathfrak{U}_i(t) )^2}{2\mathfrak{S}_i(t)}}\, \diff\mathbf{x}_i  \\
		&\quad \cdot \bbbpr{\prod\nolimits_{ j \in \{ 1,2, \dots,d \} \backslash\{i\}} \int_\R (2\pi    \mathfrak{S}_j(t))^{-\nicefrac{1}{2}}    \exp\pr*{ - \frac{(\mathbf{x}_j -\mathfrak{U}_j(t) )^2}{2\mathfrak{S}_j(t)}}\, \diff{\mathbf x_j} } \biggr] \Biggr)
		.
	\end{split}
	\end{equation}
		This 
		and the fact that
			for every
				$i \in \{1,2, \dots, d\}$,
				$t \in [0,\infty)$
			it holds that
			\begin{equation}
				\int_\R (2\pi    \mathfrak{S}_i(t))^{-\nicefrac{1}{2}}    \exp\pr*{ - \frac{( x -\mathfrak{U}_i(t) )^2}{2\mathfrak{S}_i(t)}}\, \diff x 
				=
				1
			\end{equation}
	imply that 
		for every
			$t \in [0,\infty)$,
			$x = (x_1, \dots,x_d)\in \R^d$
		it holds that 
		\begin{equation}\label{lem:derivatives4}
		\begin{split}
			&
			u(t,x)\pr*{ a(x) - \int_{\R^d} u(t,\mathbf{x})\, a(\mathbf{x}) \, \diff\mathbf{x} } 
			\\ &= 
			\frac{u(t,x)}{2} \sum_{i=1}^d \br*{- \abs{x_i}^2 + \int_\R  \abs{\mathbf{x}_i}^2 \, (2\pi    \mathfrak{S}_i(t))^{-\nicefrac{1}{2}}    \exp\pr*{ - \frac{(\mathbf{x}_i -\mathfrak{U}_i(t) )^2}{2\mathfrak{S}_i(t)}} \,\diff\mathbf{x}_i }
			.
		\end{split}
		\end{equation}
	Next observe that 
		the integral transformation theorem 
	demonstrates that 
		for every
			$i \in \{1,2, \dots,d\}$,
			$t \in [0,\infty)$
		it holds that
		\begin{equation}
		\begin{aligned}
			&
			\int_\R  x^2 \br*{ (2\pi    \mathfrak{S}_i(t))^{-\nicefrac{1}{2}}    \exp\pr*{ - \frac{( x -\mathfrak{U}_i(t) )^2}{2\mathfrak{S}_i(t)}} } \diff x 
			\\ &= \int_\R   (x + \mathfrak{U}_i(t))^2 \br*{ (2\pi    \mathfrak{S}_i(t))^{-\nicefrac{1}{2}}    \exp\pr*{ - \frac{x^2}{2\mathfrak{S}_i(t)}} } \diff x\\
			&= \int_\R   x ^2 \br*{ (2\pi    \mathfrak{S}_i(t))^{-\nicefrac{1}{2}}    \exp\pr*{ - \frac{x^2}{2\mathfrak{S}_i(t)}} }  \diff x 
			\\&\quad 
			+ \int_\R  \abs{\mathfrak{U}_i(t)}^2 \br*{ (2\pi    \mathfrak{S}_i(t))^{-\nicefrac{1}{2}}    \exp\pr*{ - \frac{x^2}{2\mathfrak{S}_i(t)}} } \diff x
			\\ &=
			\mathfrak{S}_i(t) + \abs{\mathfrak{U}_i(t)}^2
			.
		\end{aligned}
		\end{equation}
	Combining 
		this 
	with 
		\eqref{lem:derivatives4} 
	ensures that 
		for every
			$t \in [0,\infty)$,
			$x = (x_1, \dots,x_d)\in \R^d$
		it holds that 
		\begin{equation}\label{lem:derivatives3}
		\begin{split}
			&
			u(t,x)\pr*{ a(x) - \int_{\R^d} u(t,\mathbf{x})\, a(\mathbf{x}) \, \diff\mathbf{x} }
			=
			\frac{u(t,x)}{2} \sum_{i=1}^d \pr*{ \mathfrak{S}_i(t) + \abs{\mathfrak{U}_i(t)}^2 - \abs{x_i}^2 } 
			.
		\end{split}
		\end{equation}
		This 
		and \eqref{lem:derivatives2} 
	demonstrate that 
		for every
			$t \in [0,\infty)$,
			$x = (x_1, \dots,x_d)\in \R^d$
		it holds that
		\begin{equation}
		\begin{aligned}
			& u(t,x)\pr*{a(x) - \int_{D} u(t,\mathbf{x})\, a(\mathbf{x}) \, \diff\mathbf{x} } + \smallsuml_{i=1}^{d} \tfrac{1}{2} \abs{\mathfrak{m}_i}^2 \bpr{\tfrac{\partial^2 }{\partial x_i^2} u} (t,x) 
			\\ &= 
			\frac{u(t,x)}{2} \sum_{i=1}^d \br*{  \abs{\mathfrak{m}_i}^2  \pr*{ \bbbpr{\frac{x_i - \mathfrak{U}_i(t)}{\mathfrak{S}_i(t)}}^{2} - \frac{1}{\mathfrak{S}_i(t)} } + \mathfrak{S}_i(t) + \abs{\mathfrak{U}_i(t)}^2 - \abs{x_i}^2 }
			.
		\end{aligned}
		\end{equation}
	Combining 
		this 
	with 
		\eqref{lem:derivatives1*} 
	proves \cref{lem:it_eq}. 
	The proof of \cref{lem:aniso_mutator_selector} is thus complete.
\end{proof}

\subsection{Allen--Cahn PDEs with conservation of mass}
\label{subsec:allen_cahn}
In this subsection we use the machine learning-based approximation method in \cref{frame:adam}
to approximately calculate the solutions of certain Allen--Cahn PDEs with cubic nonlinearity, conservation of mass and no-flux boundary conditions (cf., e.g., Rubinstein \& Sternberg \cite{RUBINSTEIN1992}).

Assume 
	\cref{frame:adam}, 
let
	$\epsilon\in(0,\infty)$
satisfy
	$\epsilon = \tfrac{1}{10}$,
assume that
	$d\in\{1,2,5,10\}$,
	$\D = [-\nicefrac12,\nicefrac12]^d$,
	$T\in\{\nicefrac{1}{5},\nicefrac{1}{2},1\}$,
	$N=10$,
	$K_1 = K_2 = \ldots = K_N= 5$, and
	$M_1 = M_2 = \ldots = M_N = 500$,
assume that 
	$\xi^{n,m,j}, n, m, j \in \N$,
	are independent $\mathcal{U}_{D}$-distributed random variables,
assume 
	for every 
		$m \in \N$
	that
		$\gamma_m =10^{-2}$,
	and
assume 
	for every 
		$s,t \in [0,T]$, 
		$x,\mathbf{x}
		\in \D$, 
		$y,\mathbf{y} \in \R$,
		$v\in\R^d$,
		$A \in \mathcal{B}(D)$
	that
		$\nu_x(A) = \int_A \diff\mathbf{x}$,
		$g(x)= \exp (-\tfrac{1}{4}\norm{x}^2)$,
		$\mu(x)=(0, \dots, 0)$,
		$\sigma(x) v = \epsilon v$, 
		$f(t,x,\mathbf{x},y,\mathbf{y}
		)=  y - y^3 - \pr*{ \mathbf{y} - \mathbf{y}^3}$, and
	\begin{equation}
		\label{eq:Hallencahn}
		H(t,s,x,v) 
		=
		R(x,x+\mu(x)(t-s)+\sigma(x)v)
		=
		R(x,x+\epsilon v)
	\end{equation}
	(cf.\ \eqref{Y-algo-spez} and \eqref{eq:FormalXapprox}).
The solution 
	$u\colon[0,T]\times \D \to \R$ 
	of the PDE in \eqref{eq:PDE-Examples} then satisfies that 
		for every
		$t\in (0,T]$, 
		$x\in\partial_\D$
	it holds that
		$\ang*{\mathbf{n}(x) ,(\nabla_x u)(t,x)} = 0$
	and that for every
			$t\in [0,T]$, 
			$x\in\D$ 
		it holds that 
			$u(0,x)=\exp (-\tfrac{1}{4}\norm{x}^2)$ and
		\begin{equation}
			\bpr{\tfrac{\partial}{\partial t}u}(t,x)
			=
			\tfrac{\epsilon^2}{2} (\Delta_x u)(t,x) + u(t,x) - [u(t,x)]^3 - \int_{[-\nicefrac12,\nicefrac12]^d} u(t,\mathbf{x}) - [u(t,\mathbf{x})]^3 \,\diff\mathbf{x}
			.
		\label{eq:allen_cahn}
		\end{equation}
In \cref{table:allen_cahn}
we use the machine learning-based approximation method
in \cref{frame:adam}
to approximately calculate
the mean of %
$
\bV^{1,0}_N(\Theta^N_{M_N},\allowbreak (0,\ldots,0))
$,
the standard deviation of %
$
\bV^{1,0}_N(\Theta^N_{M_N},\allowbreak (0,\ldots,0))
$,
the relative $ L^1 $-approximation error associated to %
$
\bV^{1,0}_N(\Theta^N_{M_N},\allowbreak (0,\ldots,\allowbreak 0))
$,
the uncorrected sample standard deviation of the approximation error associated to %
$
\bV^{1,0}_N(\Theta^N_{M_N},\allowbreak (0,\ldots,0))
$,
and the average runtime in seconds needed for calculating one realization of $
\bV^{1,0}_N(\Theta^N_{M_N},\allowbreak (0,\ldots,0))
$
based on $5$ independent realizations (5 independent runs).
The reference value, which is used as an approximation for the unknown value $u(T,(0,\ldots,0))$
of the exact solution of \eqref{eq:allen_cahn}, has been calculated via the MLP approximation method for non-local nonlinear PDEs in \cref{frame:mlpsetting} (cf.~\cref{exampleMLP:allen_cahn} and Beck et al.~\cite[Remark~3.3]{Beck2017a}).

\begin{table}%
	\begin{center}
		\resizebox{\textwidth}{!}{\begin{approxtabular}
			$1$ & $\nicefrac{1}{5}$ & $10$ & $0.9947184$ & $0.0021832$ & $0.9932255$ & $0.0015709$ & $0.0021380$ & $31.417$\\
			$2$ & $\nicefrac{1}{5}$ & $10$ & $0.9908873$ & $0.0027061$ & $0.9868883$ & $0.0040521$ & $0.0027421$ & $35.069$\\
			$5$ & $\nicefrac{1}{5}$ & $10$ & $0.9942151$ & $0.0052064$ & $0.9710707$ & $0.0238340$ & $0.0053615$ & $38.363$\\
			$10$ & $\nicefrac{1}{5}$ & $10$ & $0.9792556$ & $0.0203935$ & $0.9514115$ & $0.0292661$ & $0.0214350$ & $42.782$\\\hline
			$1$ & $\nicefrac{1}{2}$ & $10$ & $0.9870476$ & $0.0014673$ & $0.9880013$ & $0.0014996$ & $0.0007477$ & $30.297$\\
			$2$ & $\nicefrac{1}{2}$ & $10$ & $0.9763564$ & $0.0030895$ & $0.9750274$ & $0.0024841$ & $0.0021561$ & $34.922$\\
			$5$ & $\nicefrac{1}{2}$ & $10$ & $0.9518845$ & $0.0051304$ & $0.9431354$ & $0.0092766$ & $0.0054398$ & $37.963$\\
			$10$ & $\nicefrac{1}{2}$ & $10$ & $0.9249420$ & $0.0052786$ & $0.9063239$ & $0.0205424$ & $0.0058242$ & $43.139$\\\hline
			$1$ & $1$ & $10$ & $0.9823494$ & $0.0003647$ & $0.9780817$ & $0.0043633$ & $0.0003729$ & $29.250$\\
			$2$ & $1$ & $10$ & $0.9659823$ & $0.0004128$ & $0.9658025$ & $0.0003195$ & $0.0003137$ & $34.485$\\
			$5$ & $1$ & $10$ & $0.9209547$ & $0.0019223$ & $0.9158821$ & $0.0055385$ & $0.0020988$ & $39.318$\\
			$10$ & $1$ & $10$ & $0.8693402$ & $0.0029947$ & $0.8683143$ & $0.0030165$ & $0.0015052$ & $44.258$\\			
			\hline
		\end{approxtabular}}
	\end{center}
	\caption{Numerical simulations for the approximation method in \cref{def:general_algorithm} in the case of the Allen--Cahn PDEs with conservation of mass in \eqref{eq:allen_cahn} in \cref{subsec:allen_cahn}.
	\label{table:allen_cahn}}
\end{table}

\section{{\sc Julia} source codes}
\label{sec:sourcecodes}

\subsection{General package for high-dimensional PDE approximations}
\codefile{srccode_julia/HighDimPDE.jl}
\codefile{srccode_julia/MCSample.jl}
\codefile{srccode_julia/reflect.jl}
\subsection{Implementation of the machine learning-based approximation method}
\codefile{srccode_julia/DeepSplitting.jl}
\subsection{Implementation of the multilevel Picard \texorpdfstring{ap\-prox\-i\-ma\-tion meth\-od}{approximation method}}
\codefile{srccode_julia/MLP.jl}

\subsection{{\sc Julia} source codes associated to \texorpdfstring{\cref{subsec:fisherKPP_neumann_r}}{Section 5.1}}
\codefile{srccode_julia/HighDimPDE_examples/DeepSplitting_fisherkpp_neumann.jl}
\codefile{srccode_julia/HighDimPDE_examples/MLP_fisherkpp_neumann.jl}

\subsection{{\sc Julia} source codes associated to \texorpdfstring{\cref{subsec:nonlocalcompPDE}}{Section 5.2}}
\codefile{srccode_julia/HighDimPDE_examples/DeepSplitting_nonlocal_comp.jl}
\codefile{srccode_julia/HighDimPDE_examples/MLP_nonlocal_comp.jl}

\subsection{{\sc Julia} source codes associated to \texorpdfstring{\cref{subsec:sinegordon_nonlocal}}{Section 5.3}}
\codefile{srccode_julia/HighDimPDE_examples/DeepSplitting_nonlocal_sinegordon.jl}
\codefile{srccode_julia/HighDimPDE_examples/MLP_nonlocal_sinegordon.jl}

\subsection{{\sc Julia} source codes associated to \texorpdfstring{\cref{subsec:aniso_mutator_selector}}{Section 5.4}}
\codefile{srccode_julia/HighDimPDE_examples/DeepSplitting_rep_mut.jl}
\codefile{srccode_julia/HighDimPDE_examples/MLP_rep_mut.jl}

\subsection{{\sc Julia} source codes associated to \texorpdfstring{\cref{subsec:allen_cahn}}{Section 5.5}}
\codefile{srccode_julia/HighDimPDE_examples/DeepSplitting_allencahn_neumann.jl}
\codefile{srccode_julia/HighDimPDE_examples/MLP_allencahn_neumann.jl}

\section*{Acknowledgments}

This project has been funded by the Deutsche Forschungsgemeinschaft (DFG, German Research Foundation) under Germany's Excellence Strategy EXC 2044-390685587, Mathematics M{\"u}nster: Dynamics-Geometry-Structure.
This project has been partially supported by the startup fund project of Shenzhen Research Institute of Big Data under grant No.~T00120220001.

\bibliographystyle{acm}
\bibliography{bibfile}

\end{document}